\documentstyle[amssymb,amsfonts,12pt]{amsart}

\newtheorem{theorem}{Theorem}[section]
\newtheorem{lemma}[theorem]{Lemma}

\newtheorem{proposition}[theorem]{Proposition}
\newtheorem{definition}[theorem]{Definition}

\newtheorem{question}[theorem]{Question}\newtheorem{conjecture}[theorem]{Conjecture}

\theoremstyle{remark}
\newtheorem{remark}[theorem]{Remark}
\newtheorem*{ack*}{Acknowledgment}

\def\v{{\bf v}}
\def\T{{\mathbb T}}
\def\F{{\mathcal F}}
\def\R{{\mathbb R}}
\def\E{{\mathbb E}}
\def\N{{\mathbb N}}
\def\C{{\mathbb C}}

\def\Z{{\mathbb Z}}
\def\P{{\mathcal P}}
\def\H{{\mathcal T}}
\def\W{{\mathcal L}}
\def\A{{\mathcal N}}
\def\dist{{\operatorname{dist}}}
\def\supp{{\operatorname{supp}}}
\def\bas{\begin{align*}}
\def\eas{\end{align*}}
\def\bi{\begin{itemize}}
\def\ei{\end{itemize}}
\newenvironment{proof}{\noindent {\bf Proof} }{\endprf\par}
\def \endprf{\hfill  {\vrule height6pt width6pt depth0pt}\medskip}
\def\emph#1{{\it #1}}

\begin{document}

\title[The proof of the $l^2$ Decoupling Conjecture]{The proof of the $l^2$ Decoupling Conjecture}
\author{Jean Bourgain}
\address{School of Mathematics, Institute for Advanced Study, Princeton, NJ 08540}
\email{bourgain@@math.ias.edu}
\author{Ciprian Demeter}
\address{Department of Mathematics, Indiana University, 831 East 3rd St., Bloomington IN 47405}
\email{demeterc@@indiana.edu}

\thanks{The first author is partially supported by the NSF grant DMS-1301619. The second  author is partially supported  by the NSF Grant DMS-1161752}
\begin{abstract}
We prove the $l^2$ Decoupling Conjecture for compact hypersurfaces  with  positive definite second fundamental form and also for the cone. This has a wide range  of important consequences. One of them is the validity of the Discrete Restriction Conjecture, which   implies the full range of expected $L^p_{x,t}$ Strichartz estimates for both the rational and (up to $N^\epsilon$ losses) the irrational torus. Another one is an improvement in the range for the discrete restriction theorem for lattice points on the sphere. Various applications to Additive Combinatorics, Incidence Geometry and Number Theory are also discussed.
Our argument relies on the interplay between linear and multilinear restriction theory.
\end{abstract}
\maketitle

\section{The $l^2$ Decoupling Theorem}
\bigskip

Let $S$ be a compact $C^2$ hypersurface in $\R^n$  with  positive definite second fundamental form.  Examples include the sphere $S^{n-1}$ and the truncated (elliptic) paraboloid
$$P^{n-1}:=\{(\xi_1,\ldots,\xi_{n-1},\xi_1^2+\ldots+\xi_{n-1}^2)\in\R^n:\;|\xi_i|\le 1/2\}.$$
Unless specified otherwise,  we will implicitly assume throughout the whole paper that  $n\ge 2$. We will write $A\sim B$ if $A\lesssim B$ and $B\lesssim A$. The implicit constants hidden inside the symbols  $\lesssim$ and $\sim$  will in general  depend on  fixed parameters such as $p$, $n$ and sometimes on  variable parameters such as  $\epsilon$, $\nu$. We will not record the dependence on the fixed parameters.

Let $\A_\delta$ be the $\delta$ neighborhood of $P^{n-1}$ and let $\P_\delta$ be a finitely overlapping cover of  $\A_\delta$ with curved regions $\theta$ of the form
\begin{equation}
\label{EE14}
\theta=\{(\xi_1,\ldots,\xi_{n-1},\eta+\xi_1^2+\ldots+\xi_{n-1}^2):\;(\xi_1,\ldots,\xi_{n-1})\in C_\theta,\;|\eta|\le 2\delta\},
\end{equation}
where $C_\theta$ runs over all cubes   $c+[-\frac{\delta^{1/2}}{2},\frac{\delta^{1/2}}{2}]^{n-1}$ with $c\in \frac{\delta^{1/2}}{2}\Z^{n-1}\cap [-1/2,1/2]^{n-1}$. Note that each $\theta$ sits inside a $\sim\delta^{1/2}\times \ldots\delta^{1/2}\times \delta$ rectangular box. It is also important to realize that the normals to these boxes are $\sim \delta^{1/2}$ separated. A similar decomposition exists for any  $S$ as above and we will use the same notation $\P_\delta$ for it. We will denote by $f_\theta$ the Fourier restriction of $f$ to $\theta$.

Our main result is the proof of the following {\em $l^2$ Decoupling Theorem}.
\begin{theorem}
\label{thmmm1}
Let $S$ be a compact $C^2$ hypersurface in $\R^n$  with  positive definite second fundamental form.
If $\supp(\hat{f})\subset \A_\delta$
then for $p\ge\frac{2(n+1)}{n-1}$ and $\epsilon>0$
\begin{equation}
\label{Snew26}
\|f\|_p\lesssim_\epsilon \delta^{-\frac{n-1}{4}+\frac{n+1}{2p}-\epsilon}(\sum_{\theta\in \P_\delta}\|f_\theta\|_p^2)^{1/2}.
\end{equation}

\end{theorem}

Theorem \ref{thmmm1} has been proved in \cite{GS} for $p>2+\frac8{n-1}-\frac4{n(n-1)}$. A standard construction is presented in  \cite{GS} to show that, up to the $\delta^{-\epsilon}$ term, the exponent of $\delta$ is optimal.
We point out that Wolff \cite{TWol} has initiated the study of $l^p$  decouplings, $p>2$ in the case of the cone.
His work provides part of the inspiration for our paper.

A localization argument and interpolation between $p=\frac{2(n+1)}{n-1}$ and the trivial bound for $p=2$  proves the subcritical estimate
\begin{equation}
\label{jdfgy78rtugopriguurtgpergotgi[wqxp[erirq0-wp-=ws0rf}
\|f\|_p\lesssim_\epsilon \delta^{-\epsilon}(\sum_{\theta\in \P_\delta}\|f_\theta\|_p^2)^{1/2},
\end{equation}
when $2\le p<\frac{2(n+1)}{n-1}$. Estimate \eqref{jdfgy78rtugopriguurtgpergotgi[wqxp[erirq0-wp-=ws0rf} is false for $p<2$. This can easily be seen by testing it with functions of the form $f_\theta(x)=g_\theta(x+c_\theta)$, where $\supp(\widehat{g_\theta})\subset \theta$ and the numbers $c_\theta$ are very far apart from each other.

Inequality \eqref{jdfgy78rtugopriguurtgpergotgi[wqxp[erirq0-wp-=ws0rf} has been recently proved by the first author for $p=\frac{2n}{n-1}$ in \cite{Bo2}, using a variant of the induction on scales from \cite{BG} and the multilinear restriction Theorem \ref{invthm1}.

An argument similar to the one in
 \cite{GS} was used in \cite{Dem23} to prove Theorem \ref{thmmm1} for $p>\frac{2(n+2)}{n-1}$, by interpolating Wolff's machinery with the estimate $p=\frac{2n}{n-1}$ from \cite{Bo2}. This  range is better than the one in  \cite{GS} due to the use of multilinear theory as opposed to  bilinear theory.\footnote{While both the bilinear theorem in \cite{Tao} and the multilinear theorem in \cite{BCT} are sharp, the latter one is "morally" stronger}

We mention briefly that there is a stronger form of decoupling, sometimes referred to as {\em square function estimate}, which predicts that
\begin{equation}
\label{kjfehrfyryufjhfuygyrufdhcblkskdiopweru}
\|f\|_p\lesssim_\epsilon \delta^{-\epsilon}\|(\sum_{\theta\in \P_\delta}|f_\theta|^2)^{1/2}\|_p,
\end{equation}
in the slightly smaller range $2\le p\le \frac{2n}{n-1}$. When $n=2$ this easily follows via a geometric argument. Minkowski's inequality shows that \eqref{kjfehrfyryufjhfuygyrufdhcblkskdiopweru} is indeed stronger than  \eqref{jdfgy78rtugopriguurtgpergotgi[wqxp[erirq0-wp-=ws0rf} in the range $2\le p\le \frac{2n}{n-1}$. This  is also confirmed by the lack of any results for \eqref{kjfehrfyryufjhfuygyrufdhcblkskdiopweru} when $n\ge 3$. Our methods do not seem to enable any progress on \eqref{kjfehrfyryufjhfuygyrufdhcblkskdiopweru}.

It is reasonable to hope that in the subcritical regime \eqref{jdfgy78rtugopriguurtgpergotgi[wqxp[erirq0-wp-=ws0rf}  one may be able to replace  $\delta^{-\epsilon}$  by a constant $C_{p,n}$ independent of $\delta$. This is indeed known  when $n=2$ and $p\le 4$, but seems to be in general an extremely difficult question. To the authors' knowledge, no other examples of $2<p<\frac{2(n+1)}{n-1}$ are known for when this holds.

In Section \ref{s10} we introduce a multilinear version of the decoupling inequality \eqref{Snew26} and show that the multilinear and the linear theories are essentially equivalent. This in itself is not enough to prove Theorem \ref{thmmm1}, as  Theorem \ref{invthm1} gives multilinear decoupling only in the range $2\le p\le \frac{2n}{n-1}$.
To bridge the gap between $\frac{2n}{n-1}$ and $\frac{2(n+1)}{n-1}$, in Section \ref{s8} we refine our analysis based on the multilinear theory. In particular we set up an induction on scales argument that makes use of  Theorem \ref{invthm1} at each step of the iteration, rather than once.

Let us now briefly describe some of the consequences of Theorem \ref{thmmm1}. The first one we mention is a sharp decoupling for the (truncated) cone
$$C^{n-1}=\{(\xi_1,\ldots,\xi_{n-1}, \sqrt{\xi_1^2+\ldots+\xi_{n-1}^2}),\;1\le \sqrt{\xi_1^2+\ldots+\xi_{n-1}^2}\le 2\}.$$
Abusing earlier notation, we let $\A_\delta(C^{n-1})$ be the $\delta$ neighborhood of $C^{n-1}$ and we let $\P_\delta(C^{n-1})$ be the partition of $\A_\delta(C^{n-1})$ associated with a given partition of $S^{n-1}$ into $\delta^{1/2}$ caps. More precisely, each $\theta\in \P_\delta(C^{n-1})$ is essentially a $1\times \delta\times \delta^{1/2}\times \ldots\times \delta^{1/2}$ rectangular box.

\begin{theorem}
\label{thmmm1cone}
Assume $\supp(\hat{f})\subset \A_\delta(C^{n-1})$.
Then for each $\epsilon>0$
$$
\|f\|_p\lesssim_\epsilon \delta^{-\frac{n-2}{4}+\frac{n}{2p}-\epsilon}(\sum_{\theta\in \P_\delta(C^{n-1})}\|f_\theta\|_p^2)^{1/2},\;\;\text{if }\; p\ge\frac{2n}{n-2}
$$
and
$$
\|f\|_p\lesssim_\epsilon \delta^{-\epsilon}(\sum_{\theta\in \P_\delta(C^{n-1})}\|f_\theta\|_p^2)^{1/2},\;\;\text{if }\; 2\le p\le\frac{2n}{n-2}.
$$
\end{theorem}

The proof of Theorem \ref{thmmm1cone} is presented in the last section of the paper and it turns out to be a surprisingly short application of Theorem \ref{thmmm1} for the elliptic paraboloid. It has some striking consequences, some of which were described in (and provided some of the original motivation for) the work \cite{TWol} of Wolff. Examples include progress on the ``local smoothing conjecture for the wave equation" (see \cite{Sog} and \cite{TWol}), the regularity for convolutions with arclength measures on helices \cite{PrSe}, and the boundedness
properties of the Bergman projection in tube domains over full light cones, see \cite{GarSe2} and \cite{BBGR}. We refer the interested reader to these papers for details.
\medskip

Theorem \ref{thmmm1} immediately implies the validity of the Discrete Restriction Conjecture in the expected range, see Theorem \ref{thmmmm2} below. This in turn has a wide range of interesting consequences that are detailed in Section \ref{s2}. First, we get the full range of expected $L^p_{x,t}$ Strichartz estimates for both the rational and (up to $N^\epsilon$ losses) irrational tori.
Second, we derive sharp estimates on the additive energies of various sets. These can be rephrased as incidence geometry problems and in some cases we are not aware about an alternative approach. While our theorems successfully address the case of "nicely separated" points, some intriguing questions are left open for arbitrary points.

A third type of applications includes sharp (up to $N^\epsilon$ losses) estimates for the number of solutions of various diophantine inequalities. This is rather surprising given the fact that our methods do not rely on any number theory. We believe that they provide a new angle by means of our  use of induction on scales and the topology of $\R^n$. Indeed, the Multilinear Restriction Theorem \ref{invthm1} that we use repeatedly in the proof of our main Theorem \ref{thmmm1} relies at its core on the multilinear Kakeya phenomenon, which has some topological flavor (see \cite{Guth}, \cite{Carb}).

Finally, we use Theorem \ref{thmmmm2} to improve the range from \cite{BD1}, \cite{BD2} in the discrete restriction problem for lattice points on the sphere.

In forthcoming papers we will develop the decoupling theory for arbitrary hypersurfaces with nonzero Gaussian curvature, as well as for nondegenerate curves.

\begin{ack*}
The authors are  indebted to the anonymous referees whose comments helped improving the presentation in Section 6.
The second author has benefited from helpful conversations with  Nets Katz and Andreas Seeger. The second author would like to thank his student Fangye Shi for a careful reading of the original version of the manuscript and for pointing out a few typos.
\end{ack*}

\bigskip
\section{First applications}
\label{s2}
In this section we present the first round of applications of our decoupling theory. Additional applications will appear elsewhere.

\subsection{The discrete restriction phenomenon}
\label{s21}
\bigskip
 To provide some motivation we recall the Stein-Tomas  Restriction Theorem, see \cite{SteTom}.
\begin{theorem}
Let  $S$ be  a compact $C^2$ hypersurface in $\R^n$ with nonzero Gaussian curvature and let  $d\sigma$ denote the natural surface measure on $S$. Then for $p\ge \frac{2(n+1)}{n-1}$ and $f\in L^2(S,d\sigma)$ we have
$$\|\widehat{fd\sigma}\|_{L^p(\R^n)}\lesssim\|f\|_{L^2(S)}.$$
\end{theorem}
Note that this result only needs nonzero Gaussian curvature. We will use the notation $e(a)=e^{2\pi i a}$. For fixed $p\ge \frac{2(n+1)}{n-1}$, it is an easy exercise to see that this theorem is equivalent with the statement that
$$(\frac{1}{|B_R|}\int_{B_R}|\sum_{\xi\in \Lambda}a_\xi e(\xi \cdot x)|^{p})^{1/p}\lesssim \delta^{\frac{n}{2p}-\frac{n-1}{4}}\|a_\xi\|_{l^2(\Lambda)},$$
for each $0\le \delta\le 1$, each $a_\xi\in\C$, each ball $B_R\subset \R^n$ of radius $R\sim\delta^{-1/2}$  and each $\delta^{1/2}$ separated set $\Lambda\subset S$.  Thus, the Stein-Tomas  Theorem measures the average $L^p$ oscillations of exponential sums at spatial scale equal to the inverse of the separation of the frequencies. It will be good to keep in mind that for each $R\gtrsim \delta^{-1/2}$
\begin{equation}
\label{EE43}
\|\sum_{\xi\in \Lambda}a_\xi e(\xi \cdot x)\|_{L^2(B_R)}\sim |B_R|^{1/2}\|a_\xi\|_2,
\end{equation}
as can be seen using Plancherel's Theorem.

The discrete restriction phenomenon consists in the existence of stronger cancellations  at the larger scale $R\gtrsim \delta^{-1}$. We prove the following.
\begin{theorem}
\label{thmmmm2}Let $S$ be a compact $C^2$ hypersurface in $\R^n$  with  positive definite second fundamental form.
Let $\Lambda\subset S$ be a  $\delta^{1/2}$- separated set and let $R\gtrsim \delta^{-1}$. Then for each $\epsilon>0$
\begin{equation}
\label{EE49}
(\frac{1}{|B_R|}\int_{B_R}|\sum_{\xi\in \Lambda}a_\xi e(\xi \cdot x)|^{p})^{1/p}\lesssim_{\epsilon} \delta^{\frac{n+1}{2p}-\frac{n-1}{4}-\epsilon}\|a_\xi\|_2
\end{equation}
if $p\ge \frac{2(n+1)}{n-1}$.
\end{theorem}
It has been observed in \cite{Bo2} that Theorem \ref{thmmm1} for a given $p$ implies \eqref{EE49} for the same $p$. Here is a sketch of the argument. First, note that the statement
$$\|f\|_p\lesssim \delta^{c_p}(\sum_{\theta\in \P_\delta}\|f_\theta\|_p^2)^{1/2},\;\text{whenever }\supp(\hat f)\subset \A_\delta$$
easily implies that for each $g:S\to\C$ and $R\gtrsim \delta^{-1}$
\begin{equation}
\label{EE45}
(\int_{B_R}|\widehat{gd\sigma}|^p)^{1/p}\lesssim \delta^{c_p}(\sum_{\theta\in \P_\delta}\|\widehat{g_\theta d\sigma}\|_{L^p(w_{B_{R}})}^2)^{1/2},
\end{equation}
where  here $g_\theta=g1_\theta$ is the restriction of $g$ to the $\delta^{1/2}$- cap $\theta$ on $S$. See Remark \ref{gfhgdhfghdgfhdgfhdgfhertey8we837ofhgh;kjojhiop[uolip[mlo0eru89u90e98j7897989}. Also, throughout  the paper we write
$$\|f\|_{L^p(w_{B_{R}})}=(\int_{\R^n}|f(x)|^pw_{B_{R}}(x)dx)^{1/p}$$
for weights $w_{B_R}$ which are Fourier supported in $B(0,\frac{1}{R})$ and satisfy
\begin{equation}
\label{wnret78u-0943mt7-w,-,1ir8gnrnfyqerbtf879wumweryermf,}
1_{B_R}(x)\lesssim w_{B_R}(x)\le(1+\frac{|x-c(B_R)|}{R})^{-10n}.
\end{equation}
It now suffices to use $g=\sum_{\xi\in\Lambda}a_\xi \sigma(U(\xi,\tau))^{-1}1_{U(\xi,\tau)}$ in \eqref{EE45}, where $U(\xi,\tau)$ is a $\tau$- cap on $S$ centered at $\xi$, and to let $\tau\to 0$.

Using \eqref{EE49} with $p=\frac{2(n+1)}{n-1}$ and H\"older's inequality we determine that
\begin{equation}
\label{eqinv2}
\delta^{\epsilon}\|a_\xi\|_2\lesssim_\epsilon (\frac{1}{|B_R|}\int_{B_R}|\sum_{\xi\in \Lambda}a_\xi e(\xi \cdot x)|^{p})^{1/p}\lesssim_{\epsilon} \delta^{-\epsilon}\|a_\xi\|_2
\end{equation}
for $1\le p\le \frac{2(n+1)}{n-1}$ and $R\gtrsim \delta^{-1}.$ We mention that prior to our current work, the only known results for \eqref{EE49} and \eqref{eqinv2} were the ones in the range where Theorem \ref{thmmm1} was known.
\bigskip

\bigskip

\subsection{Strichartz estimates for the classical and irrational tori}
\bigskip

The discrete restriction phenomenon has mostly been investigated in the special case when the frequency points  $\Lambda$ come from a lattice. There is extra motivation in considering this case coming from PDEs, where there is interest in establishing Strichartz estimates for the Schr\"odinger equation on the torus. Prior to the current work, the best known result for the paraboloid $$P^{n-1}(N):=\{\xi:=(\xi_1,\ldots,\xi_n)\in\Z^n:\;\xi_n=\xi_1^2+\ldots+\xi_{n-1}^2,\;|\xi_1|,\dots,|\xi_{n-1}|\le N\}$$
 was obtained by the first author \cite{Bo3}, \cite{Bo2}. We recall this result below.
\begin{theorem}[Discrete restriction: the lattice case (paraboloid)]
\label{rkjgtugfm45nbv 907u90[mq9tu8}
Let $a_\xi\in\C$ and  $\epsilon>0$. Then
\begin{enumerate}
\item[$(i)$]
if $n\ge 4$  we have
$$
\|\sum_{\xi\in P^{n-1}(N)}a_\xi e(\xi\cdot  x)\|_{L^p(\T^n)}\lesssim_\epsilon N^{\frac{n-1}{2}-\frac{n+1}{p}+\epsilon}\|a_\xi\|_{l^2},
$$
for  $p\ge \frac{2(n+2)}{n-1}$ and
$$
\|\sum_{\xi\in P^{n-1}(N)}a_\xi e(\xi\cdot  x)\|_{L^p(\T^n)}\lesssim_\epsilon N^{\epsilon}\|a_\xi\|_{l^2},
$$
for $1\le p\le \frac{2n}{n-1}$.

\item[$(ii)$] If $n=2,3$ then
$$
\|\sum_{\xi\in P^{n-1}(N)}a_\xi e(\xi\cdot  x)\|_{L^p(\T^n)}\lesssim_\epsilon N^{\epsilon}\|a_\xi\|_{l^2},
$$
for $p=\frac{2(n+1)}{n-1}$.
\end{enumerate}
\end{theorem}
The proof of (i) combines the implementation of the Stein-Tomas argument via the circle method with the inequality \eqref{jdfgy78rtugopriguurtgpergotgi[wqxp[erirq0-wp-=ws0rf} proved in \cite{Bo2}.
 The argument for (ii) is much easier, it uses the fact that circles in the plane contain "few" lattice points. It has been conjectured in \cite{Bo3} that (ii) should hold for $n\ge 4$, too. This is easily seen to be sharp, up to the $N^\epsilon$ term. We will argue below that our Theorem \ref{thmmmm2} implies this conjecture, in fact a more general version of it.

The analogous question for the more general irrational tori has been recently investigated in \cite{Bo4}, \cite{CWW}, \cite{demirb} and \cite{GOW}. More precisely, fix $\frac1{2}<\theta_1,\ldots,\theta_{n-1}<2$. For $\phi\in L^2(\T^{n-1})$  consider its Laplacian
$$\Delta \phi(x_1,\ldots,x_{n-1})=$$
$$\sum_{(\xi_1,\ldots,\xi_{n-1})\in\Z^{n-1}}(\xi_1^2\theta_1+\ldots+\xi_{n-1}^2\theta_{n-1})\hat{\phi}(\xi_1,\ldots,\xi_{n-1})e(\xi_1x_1+\ldots+\xi_{n-1}x_{n-1})$$
on the irrational torus $\prod_{i=1}^{n-1}\R/(\theta_i\Z)$. Let also
$$e^{it\Delta}\phi(x_1,\ldots,x_{n-1},t)=$$$$\sum_{(\xi_1,\ldots,\xi_{n-1})\in\Z^{n-1}}\hat{\phi}(\xi_1,\ldots,\xi_{n-1})e(x_1\xi_1+\ldots+x_{n-1}\xi_{n-1}+t(\xi_1^2\theta_1+\ldots+\xi_{n-1}^2\theta_{n-1})).$$
 We prove
\begin{theorem}[Strichartz estimates for irrational tori]
\label{thmmmm3} Let $\phi\in L^2(\T^{n-1})$ with $\supp(\hat{\phi})\subset [-N,N]^{n-1}$.
Then for each $\epsilon>0$, $p\ge \frac{2(n+1)}{n-1}$ and each interval $I\subset\R$ with $|I|\gtrsim 1$ we have
\begin{equation}
\label{EE52}
\|e^{it\Delta}\phi\|_{L^{p}(\T^{n-1}\times I)}\lesssim_{\epsilon} N^{\frac{n-1}{2}-\frac{n+1}{p}+\epsilon}|I|^{1/p}\|\phi\|_2,
\end{equation}
and the implicit constant does not depend on $I$, $N$ and $\theta_i$.
\end{theorem}
\begin{proof}
For $-N\le \xi_1,\ldots ,\xi_{n-1}\le N$ define $\eta_i=\frac{\theta_i^{1/2}\xi_i}{4N}$ and  $a_{\eta}=\hat{\phi}(\xi)$. A simple change of variables shows that
$$\int_{\T^{n-1}\times I}|e^{it\Delta}\phi|^p\lesssim $$$$\frac{1}{N^{n+1}}\int_{|y_1|,\ldots,|y_{n-1}|\le 8N \atop{y_n\in I_{N^2}}}|\sum_{\eta_1,\ldots,\eta_{n-1}}a_\eta e(y_1\eta_1+\ldots +y_{n-1}\eta_{n-1}+y_n(\eta_1^2+\ldots \eta_{n-1}^2))|^pdy_1\ldots dy_n,$$
where $I_{N^2}$ is an interval of length $\sim N^2|I|$.
By periodicity in the $y_1,\ldots,y_{n-1}$ variables we bound the above by
$$\frac{1}{N^{n+1}(N|I|)^{n-1}}\int_{B_{N^2|I|}}|\sum_{\eta_1,\ldots,\eta_{n-1}}a_\eta e(y_1\eta_1+\ldots +y_{n-1}\eta_{n-1}+y_n(\eta_1^2+\ldots \eta_{n-1}^2))|^pdy_1\ldots dy_n,$$
for some ball $B_{N^2|I|}$ of radius $\sim N^2|I|$. Our result will follow once we note that the points $$(\eta_1,\ldots,\eta_{n-1},\eta_1^2+\ldots \eta_{n-1}^2)$$ are $\sim \frac1{N}$ separated on $P^{n-1}$ and then apply Theorem \ref{thmmmm2} with $R\sim N^2|I|$.
\end{proof}
\bigskip

\begin{remark}The diagonal form $\xi_1^2\theta_1+\ldots+\xi_{n-1}^2\theta_{n-1}$ may in fact be replaced with an arbitrary definite quadratic form $Q(\xi_1,\ldots,\xi_{n-1})$, to incorporate the more general case of flat tori. The case $\theta_1=\ldots=\theta_{n-1}=1$ corresponds to the classical (periodic) torus $\T^n$. When combined with our Theorem \ref{thmmmm3}, Propositions 3.113 and 3.114 from \cite{Bo3} show that in fact  \eqref{EE52} holds true with $\epsilon=0$ in the range $p>\frac{2(n+1)}{n-1}$ for $\T^n$. Similar partial results in the direction of $\epsilon$ removal are derived for the irrational torus in \cite{GOW}.
\end{remark}

\bigskip

\subsection{The discrete restriction for lattice points on the sphere}
\bigskip

Given integers $n\ge 3$ and $\lambda=N^2\ge 1$ consider the discrete sphere
$$\F_{n,N^2}=\{\xi=(\xi_1,\ldots,\xi_n)\in\Z^n:|\xi_1|^2+\ldots|\xi_n|^2=N^2\}.$$
In \cite{Bo1}, the first author made the following conjecture  about the eigenfunctions of the Laplacian on the torus and found some partial results
\begin{conjecture}
\label{conj1}For each $n\ge 3$, $a_\xi\in \C$, $\epsilon>0$ and each $p\ge \frac{2n}{n-2}$ we have
\begin{equation}
\label{ndreuopir8ufmcsuihduyyt}
\|\sum_{\xi\in \F_{n,N^2}}a_\xi e(\xi\cdot x)\|_{L^p(\T^n)}\lesssim_\epsilon N^{\frac{n-2}{2}-\frac{n}{p}+\epsilon}\|a_\xi\|_{l^2(\F_{n,N^2})}.
\end{equation}
\end{conjecture}
We refer the reader to \cite{BD1}, \cite{BD2} for a discussion on why the critical index $\frac{2n}{n-2}$ for the sphere is different from the one for the paraboloid.
The conjecture has been verified by the authors in \cite{BD1} for $p\ge \frac{2n}{n-3}$ when $n\ge 4$  and then later improved in \cite{BD2} to $p\ge \frac{44}{7}$ when $n=4$ and $p\ge \frac{14}{3}$ when $n=5$. The methods in \cite{Bo1}, \cite{BD1} and \cite{BD2} include Number Theory of various sorts, Incidence Geometry and Fourier Analysis. Using Theorem \ref{thmmmm2} we can further improve our results.
\begin{theorem}
Let $n\ge 4$. The inequality \eqref{ndreuopir8ufmcsuihduyyt} holds for $p\ge \frac{2(n-1)}{n-3}$.
\end{theorem}
\begin{proof}
Fix $\|a_\xi\|_2=1$ and define $$F(x)=\sum_{\xi\in\F_{n,N^2}}a_\xi e(x\cdot\xi).$$
We start by recalling the following estimate (24) from \cite{BD1}, valid for $n\ge 4$ and
 $\alpha\gtrsim_\epsilon N^{\frac{n-1}{4}+\epsilon}$
\begin{equation}
\label{bnew13fnguithtuirjioju}
|\{|F|>\alpha\}|\lesssim_\epsilon \alpha^{-2\frac{n-1}{n-3}}N^{\frac2{n-3}}.
\end{equation}
By invoking interpolation with the trivial $L^\infty$ bound, it suffices to consider the endpoint $p=p_n=\frac{2(n-1)}{n-3}$. Note that $\|F\|_\infty\le N^{C_n}$. It follows that
$$\int|F|^{p_n}=\int_{N^{\frac{n-1}{4}+\epsilon}\lesssim_\epsilon\alpha\le N^{C_n}}\alpha^{p_n-1}|\{|F|>\alpha\}|d\alpha+N^{(\frac{n-1}{4}+\epsilon)(p_n-\frac{2(n+1)}{n-1})}\int|F|^{\frac{2(n+1)}{n-1}}.$$
The result will follow by applying \eqref{bnew13fnguithtuirjioju} to the first term and Theorem \ref{thmmmm2} with $p=\frac{2(n+1)}{n-1}$ to the second term.
\end{proof}

\bigskip

\subsection{Additive energies and Incidence Geometry}
\bigskip

 The proof of Theorem \ref{thmmm1} in the following sections will implicitly rely on the incidence theory of tubes and cubes. This theory manifests itself in the deep multilinear Kakeya phenomenon which lies behind  Theorem \ref{invthm1}. It thus should come as no surprise that Theorem \ref{thmmm1} has applications to Incidence Geometry.

An interesting question is whether there is a proof of Theorem \ref{thmmmm2} using softer arguments. Or at least if there is such an argument which recovers \eqref{EE49} for $R$ {\em large enough}, depending on $\Lambda$. When $n=3$ and $S=P^2$ we can prove such a result. In fact our result is surprisingly strong, in that the bound $|\Lambda|^\epsilon$ does not depend on the separation between the points in $\Lambda$.

\begin{theorem}
\label{thmmmm4}
Let $\Lambda\subset P^2$ be an arbitrary collection of distinct points.  Then for $R$ large enough, depending only on the geometry of $\Lambda$ and on its cardinality $|\Lambda|$, we have
\begin{equation}
\label{EE61}
(\frac{1}{|B_R|}\int_{B_R}|\sum_{\xi\in \Lambda}a_\xi e(\xi \cdot x)|^{4})^{1/4}\lesssim_{\epsilon} |\Lambda|^{\epsilon}\|a_\xi\|_2.
\end{equation}
\end{theorem}
Due to periodicity, this recovers (ii) of Theorem \ref{rkjgtugfm45nbv 907u90[mq9tu8} for $n=3$. To see the proof we recall some terminology and well known results.

Given an integer $k\ge 2$ and a set $\Lambda$ in $\R^n$ we introduce its $k$-energy
$$\E_k(\Lambda)=|\{(\lambda_1,\ldots,\lambda_{2k})\in \Lambda^{2k}:\;\lambda_1+\ldots+\lambda_k=\lambda_{k+1}+\ldots+\lambda_{2k}\}|.$$
Note the trivial lower bound $|\E_k(\Lambda)|\ge |\Lambda|^k$.

We recall the point-line incidence theorem due to Szemer\'edi and Trotter
\begin{theorem}[\cite{ST}]
 There are $O(|\W|+|\P|+(|\W||\P|)^{2/3})$ incidences between any collections $\W$ and $\P$ of  lines and points in the plane.
\end{theorem}
Up to extra logarithmic factors, the same thing is conjectured to hold if lines are replaced with circles. Another related conjecture is
\begin{conjecture}[The unit distance conjecture]
The number of unit distances between $N$ points in the plane is always $\lesssim_\epsilon N^{1+\epsilon}$
\end{conjecture}

The point-circle and the unit distance conjectures are thought to be rather difficult, and only partial results are known.
\bigskip

\begin{proof}[of Theorem \ref{thmmmm4}]

The following parameter encodes the "additive geometry" of $\Lambda$
$$\upsilon:=\min\left\{|\eta_1+\eta_2-\eta_3-\eta_4|:\eta_i\in \Lambda\text{ and }|\eta_1+\eta_2-\eta_3-\eta_4|\not=0\right\}.$$
We show that Theorem \ref{thmmmm4} holds if $R\gtrsim \frac{|\Lambda|^2}{\upsilon}$. Fix such an $R$.

Using restricted type interpolation  it suffices to prove
$$\frac1{|B_R|}\int_{B_R}|\sum_{\eta\in\Lambda'} e(x\cdot\eta)|^{4}dx\lesssim_\epsilon |\Lambda'|^{2+\epsilon} ,$$
for each subset $\Lambda'\subset\Lambda$. See Section 6 in \cite{BD2} for details on this type of approach.

Expanding the $L^4$ norm we need to prove
$$|\sum_{\eta_i\in\Lambda'}\frac1{R^3}\int_{B_R}e((\eta_1+\eta_2-\eta_3-\eta_4)\cdot x)dx|\lesssim_\epsilon |\Lambda'|^{2+\epsilon} .$$
Note that if $A\not=0$
$$|\int_{-R}^{R}e(At)dt|\le A^{-1}.$$
Using this we get that
$$|\sum_{\eta_i\in\Lambda'\atop{|\eta_1+\eta_2-\eta_3-\eta_4|\not=0}}\frac1{R^3}\int_{B_R}e((\eta_1+\eta_2-\eta_3-\eta_4)\cdot x)dx|\le \frac{|\Lambda'|^4}{R\upsilon} \le |\Lambda'|^2.$$

Thus it suffices to prove the following estimate  for the additive energy
\begin{equation}
\label{e1}
\E_2(\Lambda')\lesssim_{\epsilon} |\Lambda'|^{2+\epsilon}.
\end{equation}
Assume
\begin{equation}
\label{EEE1}
\eta_1+\eta_2=\eta_3+\eta_4,
\end{equation}
with $\eta_i:=(\alpha_i,\beta_i,\alpha_i^2+\beta_i^2)$.
It has been observed in \cite{Bo3} that given $A,B,C\in\R$, the equality
$$\eta_1+\eta_2=(A,B,C)$$ implies that for $i\in\{1,2\}$
\begin{equation}
\label{EEE2}
(\alpha_i-\frac{A}{2})^2+(\beta_i-\frac{B}2)^2=\frac{2C-A^2-B^2}{4}.
\end{equation}
Thus the four points $P_i=(\alpha_i,\beta_i)$ corresponding to any additive quadruple \eqref{EEE1} must belong to a  circle. As observed in \cite{Bo3}, this is enough to conclude \eqref{e1} in the lattice case, as circles of radius $M$ contain $\lesssim_\epsilon M^\epsilon$ lattice points. The bound \eqref{e1} also follows immediately if one assumes the circle-point incidence conjecture.

We need however a new observation. Note that if \eqref{EEE1} holds then in fact both $P_1,P_2$ and $P_3,P_4$ are diametrically opposite on the circle \eqref{EEE2}. Thus each additive quadruple  gives rise to a distinct right angle, the one subtended by $P_1,P_2,P_3$ (say).
The estimate \eqref{e1} is then an immediate consequence of the following application of the Szemer\'edi-Trotter Theorem.
\end{proof}

\begin{theorem}[Pach, Sharir, \cite{PS}]
\label{P-Sthm}
The number of repetitions of a given angle among $N$ points in the plane is $O(N^2\log N)$.
\end{theorem}

It has been recognized that the restriction theory for the sphere and the paraboloid are very similar\footnote{A notable difference is the lattice case of the discrete restriction, but that has to do with a rather specialized scenario}. Consequently, one expects not only Theorem \ref{thmmmm4} to be true also for $S^2$, but for a very similar argument to work in that case, too. If that is indeed the case, it does not appear to be obvious. The same argument as above shows that an additive quadruple of points on $S^2$ will belong to a circle on $S^2$,
and moreover the four points will be diametrically opposite in pairs. There will thus be at least $\E_2(\Lambda)$ right angles in $\Lambda$. This is however of no use in this setting, as $\Lambda$ lives in three dimensions. It is proved in \cite{AS} that a set of $N$ points in $\R^3$ has $O(N^{7/3})$ right angles, and moreover this bound is tight in general.

Another idea is to map an additive quadruple to the plane  using the stereographic projection. The resulting four points will again belong to a circle, so the bound on the energy would follow if the circle-point incidence conjecture is proved. Unfortunately, the stereographic  projection does not preserve the property of being diametrically opposite and thus prevents the application of Theorem \ref{P-Sthm}. We thus ask
\begin{question}
Is it true that $\E_2(\Lambda)\lesssim_\epsilon|\Lambda|^{2+\epsilon}$ for each finite $\Lambda\subset S^2$?
\end{question}
One can ask the same question for $P^{n-1}$ and $S^{n-1}$ when $n\ge 4$. The right conjecture seems to be
\begin{equation}
\label{EEE4}
\E_2(\Lambda)\lesssim_\epsilon|\Lambda|^{\frac{3n-5}{n-1}+\epsilon}.
\end{equation}
Interestingly, when $\Lambda\subset P^3$ this follows from the aforementioned result  in \cite{AS}, and in fact there is no $\Lambda^\epsilon$ loss this time. However, in the same paper \cite{AS} it is proved that this argument fails in  dimensions five and higher: there is a set with $N$ points in $\R^4$ which determines $\gtrsim N^3$ right angles. We point out that Theorem \ref{thmmmm2} implies \eqref{EEE4} for subsets of $P^{n-1}$ and $S^{n-1}$ when $n\ge 3$, in the case when the points $\Lambda$ are $\sim |\Lambda|^{-\frac{1}{n-1}}$ separated.
\bigskip

It is also natural to investigate the two dimensional phenomenon, for $S=S^1$ and $S=P^1$.
\begin{question}
\label{qqqq1}
Is it true that for each $\Lambda\subset S$
\begin{equation}
\label{dvhgu34otiu549m1uye78656w14}
\E_3(\Lambda)\lesssim_\epsilon|\Lambda|^{3+\epsilon}?
\end{equation}
\end{question}
Surprisingly, this question seems to be harder than its three dimensional analogue from Theorem \ref{thmmmm4}. Note that the case when the points are $|\Lambda|^{-C}$ separated follows from Theorem \ref{thmmmm2}. We are not aware of an alternative (softer) argument.

A positive answer to Question \ref{qqqq1} would have surprising applications to Number Theory. In particular it would answer the following  question posed in \cite{BoBo}.
\begin{question}
\label{qqqqdfkhernvvt834u2vrny45y1}
Let $N$ be a positive integer.  Does \eqref{dvhgu34otiu549m1uye78656w14} hold when $\Lambda$ are the lattice points on the circle $N^{1/2}S^1$?
\end{question}
Note that Theorem \ref{thmmmm2} is too weak to answer this question. Indeed,  rescaling by $N^{1/2}$, the lattice points in $N^{1/2}S^1$ become $N^{-1/2}$-separated points on $S^1$. However, it is known that there are $O(N^{\frac{O(1)}{\log\log N}})$ lattice points on the circle $N^{1/2}S^1$.

The analysis in \cite{BoBo} establishes some partial results as well as some intriguing connections to the theory of elliptic curves, see for example Theorem 8 there. An easier question with similar flavor is answered in the next subsection.

The best that can be said regarding Question \ref{qqqq1} with topological based methods seems to be the following
\begin{proposition}Let $S$ be either $P^1$ or $S^1$. For each $\Lambda\subset S$
$$\E_3(\Lambda)\lesssim_\epsilon|\Lambda|^{\frac72+\epsilon}.$$
\end{proposition}
\begin{proof}
This was observed  by Bombieri and the first author \cite{BoBo} when $S=S^1$. The proofs for $P^1$ and $S^1$ are very similar, we briefly sketch the details for $S=P^1$. Let $N$ be the cardinality of $\Lambda$.
It goes back to \cite{Bo3} that if
\begin{equation}
\label{EEE5}
(x_1,x_1^2)+(x_2,x_2^2)+(x_3,x_3^2)=(n,j),
\end{equation}
then the point $(3(x_1+x_2),\sqrt{3}(x_1-x_2))$ belongs to the circle centered at $(2n,0)$ and of radius squared equal to $6j-2n^2$.  Note that there are $N^2$ such points with $(x_i,x_i^2)\in \Lambda$, call this set of points $T$. Assume we have $M_n$ such circles containing roughly $2^n$ points  $(3(x_1+x_2),\sqrt{3}(x_1-x_2))\in T$ in such a way that
\eqref{EEE5} is satisfied for some $x_3\in \Lambda$. Then clearly
$$\E_3(S)\lesssim \sum_{2^n\le N}M_n2^{2n}.$$

It is easy to see that
\begin{equation}
\label{EEE6}
M_n2^n\lesssim N^3,
\end{equation}
as each point in $T$ can belong to at most $N$ circles.

The nontrivial estimate is
\begin{equation}
\label{EEE7}
M_n2^{3n}\lesssim N^4,
\end{equation}
which is an immediate consequence of the Szemer\'edi-Trotter Theorem for curves  satisfying the following two fundamental axioms: two curves intersect in $O(1)$ points, and there are $O(1)$ curves passing through any two given points. The number of incidences between such curves and points is the same as in the case of lines and points, see for example Theorem 8.10 in \cite{TV}. Note that since our circles have centers on the $x$ axis, any two points in $T$ sitting in the upper (or lower) half plane determine a unique circle. Combining the two inequalities we get for each $n$
$$M_n2^{2n}\lesssim N^{\frac72}.$$
\end{proof}
\bigskip

In the case when $\Lambda\subset S^1$, the same argument leads to incidences between unit circles and points. The outcome is the same, since for any two points there are at most two unit circles passing through them. An interesting observation is the fact that Question \ref{qqqq1} has a positive answer if the Unit Distance Conjecture is assumed. Indeed, the argument above presents us with a collection $T$ of $N^2$ points and a collection  of $\lesssim N^3$ unit circles. For $2^n\lesssim N$ let $M_n$ be the number of such circles  with $\sim 2^n$ points. There will be at least $M_n2^n$ unit distances among the $N^2$ points and the $M_n$ centers. The Unit Distances Conjecture forces $M_n2^n\lesssim_\epsilon(M_n+N^2)^{1+\epsilon}$. Since $M_n\lesssim N^3$, it immediately follows that $M_n2^{2n}\lesssim_\epsilon N^{3+\epsilon}$ which gives the desired bound on the energy.

It seems likely that in order to achieve the conjectured bound on $\E_3(\Lambda)$, the structure of $T$ must be exploited, paving the way to algebraic methods. One possibility is to make use of the fact that $T$ has sumset structure.   Another interesting angle for the parabola is the following. Recall that whenever \eqref{EEE5} holds, the three points $(3(x_i+x_j),\sqrt{3}(x_i-x_j))$, $(i,j)\in\{(1,2),(2,3),(3,1)\}$, belong to the circle centered at $(2n,0)$ and of radius squared equal to $6j-2n^2$. One can easily check that if fact they form an equilateral triangle! This potentially opens up the new toolbox of symmetries, since, for example, the rotation by $\pi/3$ about the center of any such circle $C$ will preserve $C\cap T$.
\bigskip

\subsection{Additive energies of annular sets}
We start by mentioning a more general version of Theorem \ref{thmmmm2}.

\begin{theorem}
\label{thmmmmmmmmmmmmmmmmmmmmmmmmmmmmmmmmmmmmm3}
Let $S$ be a $C^2$ compact hypersurface in $\R^n$ with positive definite second fundamental form. For each $\theta\in \P_\delta$ let $\Lambda_\theta$ be a collection of  points in $\theta$ and let $\Lambda=\cup_{\theta}\Lambda_\theta$. Then for each $R$- ball $B_{R}$ with $R\gtrsim \delta^{-1}$ we have
$$\|\sum_{\xi\in\Lambda}a_\xi e(x\cdot\xi)\|_{L^{\frac{2(n+1)}{n-1}}(B_{R})}\lesssim_\epsilon \delta^{-\epsilon}(\sum_\theta\|\sum_{\xi\in\Lambda_\theta}a_\xi e(x\cdot\xi)\|_{L^{\frac{2(n+1)}{n-1}}(w_{B_R})}^2)^{1/2}.$$
\end{theorem}
To see why this holds, note first that the case $R\sim\delta^{-1}$ follows by applying (the localized version of) Theorem \ref{thmmm1} to functions whose Fourier transforms approximate weighted sums of Dirac deltas supported on $\Lambda$. The case $R\gtrsim \delta^{-1}$ then follows using Minkowski's inequality.

For  $R>1$ define
$$A_R=\{\xi\in\R^2:\;R\le |\xi|\le R+R^{-1/3}\}$$
and $A_R'=A_R\cap \Z^2$. We prove the following inequality related to Question \ref{qqqqdfkhernvvt834u2vrny45y1}.
\begin{theorem}
\begin{equation}
\label{invveq1}
\E_3(A_R')\lesssim_\epsilon |A_R'|^{3+\epsilon}.
\end{equation}
\end{theorem}
Note that this is essentially sharp.
The old Van der Corput estimate
$$|N(R)-\pi R^2|=o(R^{2/3})$$
for the error term in the Gauss circle problem shows that $|A_R'|=2\pi R^{2/3}+o(R^{2/3})$. It thus suffices to show
$$\|\sum_{\xi\in A_R'}e(\xi\cdot x)\|_{L^6(\T^2)}\lesssim_\epsilon R^{\frac13+\epsilon}.$$
Subdivide $A_R$ into sectors $A_\alpha$ of size $\sim R^{1/3}\times R^{-1/3},$ so that each of them fits inside a rectangle $R_\alpha$ of area  $<\frac12$.  Applying Theorem \ref{thmmmmmmmmmmmmmmmmmmmmmmmmmmmmmmmmmmmmm3} after rescaling by $R$ and using periodicity we get
\begin{equation}
\label{kpty90-u8ivinhb8muviuhyj,[c vhgbphuiot579}
\|\sum_{\xi\in A_R'}e(\xi\cdot x)\|_{L^6(\T^2)}\lesssim_\epsilon R^\epsilon(\sum_\alpha\|\sum_{\xi\in A_\alpha'}e(\xi\cdot x)\|_{L^6(\T^2)}^2)^{1/2},
\end{equation}
with $A_\alpha'=A_\alpha\cap \Z^2$.

An elementary observation which goes back (at least) to Jarn\'ik's work \cite{Jar} is the fact that the area determined by a nondegenerate triangle with vertices in $\Z^2$ is half an integer. It follows that the points in each $A_\alpha'$ lie on a line $L_\alpha$. In fact they must be equidistant, with consecutive points at distant $d$, for some $d\ge 1$. Define now for $1\le 2^s\lesssim R^{1/3}$
$$\W_s=\{\alpha:2^s\le |A_\alpha'|<2^{s+1}\}.$$
Let also $\W_{s,m}$ be those $\alpha\in \W_s$ for which $2^m\le d<2^{m+1}$. Note that if $\alpha\in\W_{s,m}$ then  $L_\alpha$ makes an angle $\sim 2^{-m}2^{-s}R^{-1/3}$ with the long axis of $R_\alpha$. Thus the directions of the lines $L_\alpha$ will be distinct for each collection of $\alpha\in\W_{s,m}$ whose corresponding arcs on $S^1$ are $C2^{-m}2^{-s}R^{-1/3}$-separated.
Obviously there are $u,v\in A_\alpha'$ such that $|u-v|\sim 2^m$. Since there are $O(2^{2m})$ lattice points with length $\sim 2^m$, it follows that there can  be at most $O(2^{2m})$ elements $\alpha$ in $\W_{s,m}$ which are $C2^{-m}2^{-s}R^{-1/3}$-separated. Thus $|\W_{s,m}|\lesssim 2^{2m}R^{1/3}2^{-m}2^{-s}$. As $2^{m+s}\lesssim R^{1/3}$ we conclude that $|\W_s|\lesssim R^{2/3}2^{-2s}$.

Note that by using H\"older's inequality with $L^2-L^\infty$ endpoints we have
$$\|\sum_{\xi\in A_\alpha'}e(\xi\cdot x)\|_{L^6(\T^2)}\le |A_\alpha'|^{5/6}.$$
Using this, the bound on $|\W_s|$ and \eqref{kpty90-u8ivinhb8muviuhyj,[c vhgbphuiot579} finishes the proof of \eqref{invveq1}.

\subsection{Counting solutions of Diophantine inequalities}
In this section we show how to use the Decoupling Theorem to recover and generalize results from the literature as well as to prove some new type of results. We do not aim at providing a systematic study of these problems but rather to explain the way our methods become useful in this context.

To motivate our first application  we consider the system of equations for $k\ge 2$
$$\begin{cases}n_1^k+n_2^k+n_3^k&=n_4^k+n_5^k+n_6^k \\ \hfill n_1+n_2+n_3&=n_4+n_5+n_6  \end{cases},$$
with $1\le n_i\le N$.
It is easy to see that there are $6N^3$ trivial solutions. The question here is to determine the correct asymptotic for the number $U_k(N)$ of nontrivial solutions. This is in part motivated by connections to the Waring problem, see \cite{Bok}. The case $k=3$ known as the Segre cubic has been intensely studied. Vaughan and Wooley have proved in \cite{VauWol} that $U_3(N)\sim N^2(\log N)^5$ , see also \cite{Bre} for a more precise result. For $k\ge 4$, Greaves \cite{Gre} (see also \cite{SkW}) has proved that $U_k(N)=O(N^{\frac{17}{6}+\epsilon})$. All these results follow through the use of rather delicate Number Theory.

While our methods in this paper can not produce such fine estimates, they successfully address the perturbed case. The following result is perhaps a surprising consequence of Theorem \ref{thmmmm2}.
\begin{theorem}
\label{thmmnumberofsoll}
For fixed $k\ge 2$ and $C$ the system
$$\begin{cases}&|n_1^k+n_2^k+n_3^k-n_4^k-n_5^k-n_6^k|\le CN^{k-2} \\  \hfill &n_1+n_2+n_3=n_4+n_5+n_6\end{cases}$$
has  $O(N^{3+\epsilon})$ solutions with $n_i\sim N$.
\end{theorem}

\begin{proof}
Apply Theorem \ref{thmmmm2} to the curve $$\{(\xi,\xi^k):\;|\xi|\sim 1\},$$ the points $$\Lambda=\{(\frac{n}{N},(\frac{n}{N})^k):\;n\sim N\}$$
 and $\delta=N^{-2}$. We get that
 $$\frac1{N^4}\int_{|x|\le N^2}\int_{|y|\le N^2}|\sum_{n\sim  N}e(x\frac{n}{N}+y(\frac{n}{N})^k)|^6dxdy\lesssim_\epsilon N^{3+\epsilon}.$$
 Upon rescaling and using periodicity we get
 $$N^{k-3}\int_{|x|\le N}\int_{|y|\le N^{2-k}}|\sum_{n\sim  N}e(x n+y n^k)|^6dxdy=$$
 \begin{equation}
 \label{dttty555534ssde}
 N^{k-2}\int_{|x|\le 1}\int_{|y|\le N^{2-k}}|\sum_{n\sim  N}e(x n+y n^k)|^6dxdy\lesssim_\epsilon N^{3+\epsilon}.
 \end{equation}
 Let now $\phi:\R\to [0,\infty)$ be a positive Schwartz function with  positive Fourier transform satisfying $\widehat{\phi}(\xi)\gtrsim 1$ for $|\xi|\le 1$. Define $\phi_{N}(y)=\phi(N^{k-2}y)$. A standard argument allows us to replace the cutoff $|y|\le N^{2-k}$ with $\phi_{N}(y)$ in \eqref{dttty555534ssde}. It suffices then to note that
 $$N^{k-2}\int_{|x|\le 1}\int_\R|\sum_{n\sim  N}e(x n+y n^k)|^6\phi_N(y)dxdy=$$$$\sum_{n_i\sim N\atop{n_1+n_2+n_3=n_4+n_5+n_6}}\widehat{\phi}(N^{2-k}(n_1^k+n_2^k+n_3^k-n_4^k-n_5^k-n_6^k)).$$
 \end{proof}

Note also that our method proves that
$$N^{k-2}\int_{|x|\le 1}\int_{|y-c|\le N^{2-k}}|\sum_{n\sim  N}e(x n+y n^k)|^6dxdy\lesssim_\epsilon N^{3+\epsilon},$$
for each $c\in\R$.
The difficulty in proving this for $k\ge 3$ using purely number theoretic methods comes from estimating the contribution of the minor arcs.  When $k=2$ the left hand side is at least $cN^3\log N$, which shows that one can not dispense with the $N^\epsilon$ term.  This can be seen by evaluating the contribution from the major arcs, see for example page 118 in \cite{Bo3}.
\bigskip

Our second application generalizes the result from \cite{RoSa} ($k=4$) to $k\ge 4$. Its original motivation lies in the study of the Riemann zeta function on the critical line (cf. \cite{BoIw1}, \cite{BoIw2}) and also in getting refinements of Heath-Brown's variant of Weyl's inequality, see \cite{RoSa}.
\begin{theorem}
\label{ggg555rrrttt778ghtr}
For $k\ge 4$ and $0\le \lambda\le 1$ we have
$$\int_{|x|\le 1}\int_{0}^{\lambda}|\sum_{n\sim N}e(xn^2+yn^k)|^6dxdy\lesssim_{\epsilon}\lambda N^{3+\epsilon}+N^{4-k+\epsilon}.$$
In particular, the system
$$\begin{cases}&|n_1^k+n_2^k+n_3^k-n_4^k-n_5^k-n_6^k|\le CN^{k-1} \\  \hfill &n_1^2+n_2^2+n_3^2=n_4^2+n_5^2+n_6^2\end{cases}$$
has  $O(N^{3+\epsilon})$ solutions with $n_i\sim N$.
\end{theorem}
\begin{proof}
The estimate on the number of solutions follows by using $\lambda=N^{1-k}$. Note that it suffices to prove that
$$\int_{|x|\le 1}\int_{J}|\sum_{n\sim N}e(xn^2+yn^k)|^6dxdy\lesssim_{\epsilon}N^{4-k+\epsilon},$$
for each interval $J$ with length $N^{1-k}$.

We apply Theorem \ref{thmmmmmmmmmmmmmmmmmmmmmmmmmmmmmmmmmmmmm3} to the curve
$$\{(\xi^2,\xi^k):\;|\xi|\sim 1\},$$ the points $$\Lambda=\{((\frac{n}{N})^2,(\frac{n}{N})^k):\;n\sim N\},$$
 $R^{-1}=\delta=N^{-1}$ and $B_N=[MN,(M+1)N]\times N^kJ$ with $M\in\{-N,\ldots,0,\ldots, N-1\}$. Summing over $M$ we get due to  periodicity
$$\|\sum_{n\sim N}e(x'\frac{n^2}{N^2}+y'\frac{n^k}{N^k})\|_{L^6(|x'|\le N^2,\;y'\in N^kJ)}\lesssim_{\epsilon}$$$$N^\epsilon(\sum_\alpha\|\sum_{n\in I_\alpha}e(x'\frac{n^2}{N^2}+y'\frac{n^k}{N^k})\|_{L^6(|x'|\le N^2,\;y'\in N^kJ)}^2)^{1/2}.$$
Here $I_\alpha=[n_\alpha,n_\alpha+N^{1/2}]$ are intervals of length $N^{1/2}$ that partition the integers $n\sim N$. It follows after a change of variables that

$$\|\sum_{n\sim N}e(x{n^2}+y{n^k})\|_{L^6(|x|\le 1,\;y\in J)}\lesssim_{\epsilon}$$
\begin{equation}
\label{juy45t7f904tbg707f8-dk90f78gr5jtu90fjrtjho}
N^\epsilon(\sum_\alpha\|\sum_{n\in I_\alpha}e(x{n^2}+y{n^k})\|_{L^6(|x|\le 1,\;y\in J)}^2)^{1/2}.
\end{equation}
Next note that for $y\in J$
$$|\sum_{n\in I_\alpha}e(x{n^2}+y{n^k})|=$$
$$|\sum_{m=1}^{N^{1/2}}c_{m,J,n_\alpha}e(m^2(x+\frac{k(k-1)}{2}n_\alpha^{k-2}y)+m(2xn_\alpha+kn_\alpha^{k-1}y))|+O(1),$$
with $|c_{m,J,n_\alpha}|=1$.  To estimate the first term we change variables to
$$\begin{cases}&x'=x +\frac{k(k-1)}{2}n_\alpha^{k-2}y \\ &y'=(2k-k^2)n_\alpha^{k-1}y \hfill \end{cases}.$$
We get
$$\|\sum_{n\in I_\alpha}e(x{n^2}+y{n^k})\|_{L^6(|x|\le 1,\;y\in J)}\lesssim $$$$n_\alpha^{-\frac{k-1}{6}} \|\sum_{m=1}^{N^{1/2}}c_{m,J,n_\alpha}e(x'm^2+2x'n_\alpha m+ my')\|_{L^6(B_C)}+O(N^{\frac{1-k}{6}})=$$
$$n_\alpha^{-\frac{k-1}{6}} \|\sum_{m=1+n_\alpha}^{N^{1/2}+n_\alpha}c_{m,J,n_\alpha}e(x'm^2+my')\|_{L^6(B_C)}+O(N^{\frac{1-k}{6}}),$$
for some ball $B_C$ of radius $C=O(1)$.
This can further be seen to be $O(N^{\frac14+\frac{1-k}{6}+\epsilon})$ by the result in Theorem \ref{rkjgtugfm45nbv 907u90[mq9tu8}. We conclude that  \eqref{juy45t7f904tbg707f8-dk90f78gr5jtu90fjrtjho} is $O(N^{\frac12+\frac{1-k}{6}+\epsilon})$, as desired.

\end{proof}

There are further number theoretical consequences of the decoupling theory that will be investigated elsewhere.

\section{Norms and wave packet decompositions}\label{s3}
\bigskip

We will use $C$ to denote various constants that are allowed to depend on the fixed parameters $n,p$, but never on the scale $\delta$.
 $|\cdot|$ will denote both the Lebesgue measure on $\R^n$ and the cardinality of finite sets.

This section and the next one is concerned with introducing some of the tools that will be used in the proof of Theorem \ref{thmmm1} from Section \ref{s8}.
For $2\le p\le \infty$  we define the norm
$$\|f\|_{p,\delta}=(\sum_{\theta\in\P_\delta}\|f_\theta\|_p^2)^{1/2},$$
where $f_\theta$ is the Fourier restriction of $f$ to $\theta$.
We note the following immediate consequence of H\"older's inequality
\begin{equation}
\label{EE1}
\|f\|_{p,\delta}\le \|f\|_{2,\delta}^{\frac2p}\|f\|_{\infty,\delta}^{1-\frac2p}
\end{equation}
and the fact that if $\supp(\hat{f})\subset \A_\delta$ then
$$\|f\|_{2,\delta}\sim\|f\|_2.$$

\begin{definition}
Let $N$ be a real number greater than 1. An $N$-tube $T$ is an $N^{1/2}\times\ldots \times N^{1/2}\times N$  rectangular parallelepiped in $\R^n$ which has dual orientation to some $\theta=\theta(T)\in \P_\delta$. We call a collection of $N$-tubes  separated if no more than $C$ tubes with a given orientation overlap.
\end{definition}
Let $\phi:\R^n\to\R$ be given by
$$\phi(x)=(1+|x|^2)^{-M},$$
for some $M$ large enough compared to $n$, whose value will become clear from the argument.
Define $\phi_T=\phi\circ a_T$, where $a_T$ is the affine function mapping $T$ to the unit cube in $\R^n$ centered at the origin.

\begin{definition}
An $N$-function is a function $f:\R^n\to \C$ such that
$$f=\sum_{T\in \H(f)}f_T$$
where $\H(f)$ consists of finitely many separated $N$-tubes $T$ and moreover
$$|f_T|\le \phi_T,$$
$$\|f_T\|_p\sim |T|^{1/p},\;1\le p\le\infty$$
and
$$\supp (\widehat{f_T})\subset \theta(T).$$

For $\theta\in \P_{1/N}$ let
$\H(f,\theta)$ denote the $N$-tubes in $\H(f)$ dual to $\theta$ .
An $N$-function is called balanced if
$|\H(f,\theta)|\le 2 |\H(f,\theta')|$
whenever $\H(f,\theta), \H(f,\theta')\not=\emptyset.$
\end{definition}
The $\|\cdot\|_{p,\delta}$ norms of $N$-functions are asymptotically determined by their  plate distribution over the sectors $\theta$.
\begin{lemma}
For each $N$-function $f$ and for  $2\le p\le \infty$
\begin{equation}
\label{EE2}
\|f\|_{p,1/N}\sim N^{\frac{n+1}{2p}}(\sum_{\theta}|\H(f,\theta)|^{\frac2p})^{1/2}.
\end{equation}
If the $N$-function is balanced then
\begin{equation}
\label{EE24}
\|f\|_{p,1/N}\sim N^{\frac{n+1}{2p}}M(f)^{\frac{1}{2}-\frac1p}|\H(f)|^{1/p},
\end{equation}
where $M(f)$ is the number of sectors $\theta$ for which $\H(f,\theta)\not=\emptyset$.
\end{lemma}
\begin{proof}
It suffices to prove \eqref{EE2} when  $\H(f)=\H(f,\theta)$ for some $\theta$. We first observe the trivial estimates $\|f\|_1\lesssim |T||\H(f)|$, $\|f\|_\infty\lesssim 1$ and $\|f\|_2\sim |T|^{1/2}|\H(f)|^{1/2}$. Applying H\"older's inequality twice we get
$$\|f\|_2^{\frac{2(p-1)}{p}}\|f\|_1^{\frac{2-p}{p}}\le\|f\|_p\le \|f\|_1^{1/p}\|f\|_\infty^{1/p'},$$
which is exactly what we want.
\end{proof}
The crucial role played by balanced $N$-functions is encoded by
\begin{lemma}
\label{lemmabalancednfunc}

(i) Each $N$-function $f$ can be written as the sum of $O(\log|\H(f)|)$ balanced $N$-functions.

(ii) For each balanced $N$-function $f$ and $2\le p\le \infty$ we have the converse of \eqref{EE1}, namely
\begin{equation}
\label{EE10}
\|f\|_{p,1/N} \sim \|f\|_{2,1/N}^{\frac2p}\|f\|_{\infty,1/N}^{1-\frac2p}.
\end{equation}
\end{lemma}
\begin{proof}
Note that $(i)$ is immediate by using dyadic ranges. Also, $(ii)$ will follow from \eqref{EE24}.
\end{proof}

In the remaining sections we will  use the fact that the contribution of $f$ to various inequalities comes from   logarithmically many $N$-functions. The basic mechanism is the following.

\begin{lemma}[Wave packet decomposition]
\label{WPD}
Assume $f$ is Fourier supported in $\A_\delta$. Then for each dyadic $0<\lambda\lesssim \|f\|_{\infty,\delta}$ there is an $N=\delta^{-1}$-function $f_\lambda$ such that
$$f=\sum_{\lambda\lesssim \|f\|_{\infty,\delta}}\lambda f_\lambda$$
and for each $2\le p<\infty$ we have
\begin{equation}
\label{EE8}
\lambda^pN^{\frac{n+1}{2}}|\H(f_\lambda)|\le \|\lambda f_\lambda\|_{p,\delta}^p\lesssim \|f\|_{p,\delta}^p.
\end{equation}
\end{lemma}
\begin{proof}
Using a partition of unity write
$$f=\sum_{\theta\in \P_\delta}\tilde{f}_\theta$$
with $\tilde{f}_\theta=f_\theta*K_\theta$ Fourier supported in $\frac9{10}\theta$ with $\|K_\theta\|_1\lesssim 1$.
Consider a windowed Fourier series expansion for each $\tilde{f}_\theta$
$$\tilde{f}_\theta=\sum_{T\in \H_\theta}\langle \tilde{f}_\theta,\varphi_T\rangle\varphi_T,$$
where $\varphi_T$ are $L^2$ normalized Schwartz functions Fourier localized in $\theta$ such that $$|T|^{1/2}|\varphi_T|\lesssim \phi_T.$$
The tubes in $\H_\theta$ are separated.
Note that by H\"older's inequality
$$|a_T:=\frac{1}{|T|^{1/2}}\langle \tilde{f}_\theta,\varphi_T\rangle|\lesssim \|\tilde{f}_\theta\|_\infty\lesssim \|f_\theta\|_\infty\le \|f\|_{\infty,\delta}.$$
It is now clear that we should take
$$f_\lambda=\sum_{\theta}\sum_{T\in\H_\theta:\;|a_T|\sim\lambda}a_T\lambda^{-1}|T|^{1/2}\varphi_T.$$

To see \eqref{EE8} note that the first inequality follows from \eqref{EE2} and the fact that $\|\cdot\|_{l^{p/2}}\le \|\cdot\|_{l^{1}}$. To derive the second inequality, it suffices to prove that for each $\theta$
$$\|\sum_{T\in\H_\theta:\;|a_T|\sim\lambda}\langle \tilde{f}_\theta,\varphi_T\rangle\varphi_T\|_p\lesssim\|f_\theta\|_p.$$
Using \eqref{EE2} and the immediate consequence of H\"older's inequality $|a_T|^p\lesssim\int|\tilde{f}_\theta|^p|T|^{-1/2}\varphi_T$ we get
$$\|\sum_{T\in\H_\theta:\;|a_T|\sim\lambda}\langle \tilde{f}_\theta,\varphi_T\rangle\varphi_T\|_p^p\lesssim \lambda^p|T||\{T\in\H_\theta:\;|a_T|\sim\lambda\}|\lesssim$$
$$\lesssim |T|\sum_{T\in\H_\theta}|a_T|^p\lesssim \int|\tilde{f}_\theta|^p\sum_{T\in\H_\theta}\phi_T\lesssim \int|\tilde{f}_\theta|^p\lesssim \int|{f}_\theta|^p.$$
\end{proof}
\bigskip

\section{Parabolic rescaling}\label{s5}
\bigskip

\begin{proposition}
\label{propo:parabooorescal}
Let $\delta\le \sigma<\frac12$ and  $K_p(\frac\delta\sigma)$ be such that
$$\|f\|_p\le K_p(\frac\delta\sigma) (\sum_{\theta\in \P_{\frac\delta\sigma}}\|f_\theta\|_p^2)^{1/2},$$
for each $f$ with Fourier support in $\A_{\frac\delta\sigma}$.
Then for each $f$ with Fourier support in $\A_\delta$ and for each $\tau\in \P_\sigma$ we have
$$\|f_\tau\|_p\lesssim K_p(\frac\delta\sigma)(\sum_{\theta\in \P_\delta:\:\theta\cap \tau\not=\emptyset}\|f_\theta\|_p^2)^{1/2}.$$
\end{proposition}
\begin{proof}
Let $a=(a_1,\ldots,a_{n-1})$ be the center of the  $\sigma^{1/2}$-cube $C_\tau$, see \eqref{EE14}. We will perform the parabolic rescaling via the affine transformation
$$L_\tau(\xi_1,\ldots,\xi_n)=(\xi_1',\ldots,\xi_n')=(\frac{\xi_1-a_1}{\sigma^{1/2}},\ldots,\frac{\xi_{n-1}-a_{n-1}}{\sigma^{1/2}}, \frac{\xi_n-2\sum_{i=1}^{n-1}a_i\xi_i+\sum_{i=1}^{n-1}a_i^2}{\sigma}).$$
Note that
$$\xi_n'-\sum_{i=1}^{n-1}{\xi_i'}^2=\sigma^{-1}(\xi_n-\sum_{i=1}^{n-1}\xi_i^2).$$
It follows that $L_\tau$ maps the Fourier support $\A_\delta\cap \tau$ of $f_\tau$ to $\A_{\frac\delta\sigma}\cap ([-\frac12,\frac12]^{n-1}\times\R)$. Also, for each $\tau'\in \P_{\frac\delta\sigma}$ we have that $L_\tau(\theta)=\tau'$ for some $\theta\in \P_\delta$ with $\theta\cap \tau\not =\emptyset$. Thus
$$\|f_\tau\|_p^p=\|g\|_p^p(\operatorname{det}(L_\tau))^{1-p},$$
where $g$ is the $L_\tau$ dilation of $f_\tau$ Fourier supported in $\A_{\frac\delta\sigma}\cap ([-\frac12,\frac12]^{n-1}\times\R)$.
By invoking the hypothesis we get that
$$\|g\|_p\lesssim K_p(\frac\delta\sigma)(\sum_{\tau'\in \P_{\frac\delta\sigma}}\|g_{\tau'}\|_p^2)^{1/2}.$$
We are done if we use  the fact that
$$\|f_{\theta}\|_p^p=\|g_{\tau'}\|_p^p(\operatorname{det}(L_\tau))^{1-p}$$
whenever $L_\tau(\theta)=\tau'$.
\end{proof}

\section{Linear versus multilinear decoupling}

\label{s10}

The material in this section is an application of the Bourgain-Guth induction on scales \cite{BG} and it is most closely connected to the argument in \cite{Bo2}.
Let $g:P^{n-1}\to\C$. For a cap $\tau$ on $P^{n-1}$ we let $g_\tau=g1_\tau$ be the (spatial) restriction of $g$ to $\tau$.
We denote by $\pi:P^{n-1}\to [-1/2,1/2]^{n-1}$ the projection map.
\begin{definition}
We say that the caps $\tau_1,\ldots,\tau_n$ on $P^{n-1}$ are $\nu$-transverse if the volume of the parallelepiped spanned by any unit normals $v_i$ at $\tau_i$ is greater than $\nu$.
\end{definition}

We denote by $C_{p,n}(\delta,\nu)$ the smallest constant such that
$$\|(\prod_{i=1}^n|\widehat{g_{\tau_i}d\sigma}|)^{1/n}\|_{L^p(B_{\delta^{-1}})}\le C_{p,n}(\delta,\nu)\left[\prod_{i=1}^n(\sum_{\theta:\;\delta^{1/2}-\text{cap}\atop{\theta\subset\tau_i}}\|\widehat{g_{\theta}d\sigma}\|_{L^p(w_{B_{\delta^{-1}}})}^2)^{1/2}\right]^{1/n},$$
for each $\nu$-transverse caps $\tau_i\subset P^{n-1}$, each $\delta^{-1}$ ball $B_{\delta^{-1}}$ and each $g:P^{n-1}\to\C$.

Let also $K_{p,n}(\delta)$ be the smallest constant such that
$$\|\widehat{gd\sigma}\|_{L^p(B_{\delta^{-1}})}\le K_{p,n}(\delta)(\sum_{\theta:\delta^{1/2}-\text{cap}}\|\widehat{g_{\theta}d\sigma}\|_{L^p(w_{B_{\delta^{-1}}})}^2)^{1/2},$$
for each $g:P^{n-1}\to\C$ and each $\delta^{-1}$ ball $B_{\delta^{-1}}$.
\begin{remark}
\label{gfhgdhfghdgfhdgfhdgfhertey8we837ofhgh;kjojhiop[uolip[mlo0eru89u90e98j7897989}
As before,  the norm $\|f\|_{L^p(w_{B_R})}$ refers to the weighted $L^p$ integral $$(\int_{\R^n}|f(x)|^pw_{B_R}(x)dx)^{1/p}$$
for some weight satisfying \eqref{wnret78u-0943mt7-w,-,1ir8gnrnfyqerbtf879wumweryermf,}.
It is important to realize that there are such weights which in addition are Fourier supported in $B(0,R^{-1})$. Note also that if $g$ is supported on $P^{n-1}$ and if $\widehat{w_{B_R}}$ is supported in $B(0,R^{-1})$, then $(\widehat{gd\sigma})w_{B_R}$ has Fourier support  inside $\A_{R^{-1}}$. This simple observation justifies the various (entirely routine) localization arguments that follow, as well as the interplay between Fourier transforms of functions and Fourier transforms of measures supported on $P^{n-1}$. In particular let $K_{p,n}^{(1)}(\delta)$ be the smallest constant such that
$$\|\widehat{gd\sigma}\|_{L^p(w_{B_{\delta^{-1}}})}\le K_{p,n}^{(1)}(\delta)(\sum_{\theta:\delta^{1/2}-\text{cap}}\|\widehat{g_{\theta}d\sigma}\|_{L^p(w_{B_{\delta^{-1}})}}^2)^{1/2},$$
for each $g:P^{n-1}\to\C$ and each $\delta^{-1}$ ball $B_{\delta^{-1}}$.
Then $ K_{p,n}^{(1)}(\delta)\sim_{n,p} K_{p,n}(\delta)$.
Also, if $ K_{p,n}^{(2)}(\delta)$, $ K_{p,n}^{(3)}(\delta)$, $ K_{p,n}^{(4)}(\delta)$ are the smallest constants such that
$$
\|f\|_{L^p(\R^n)}\le K_{p,n}^{(2)}(\delta)(\sum_{\theta\in \P_\delta}\|f_\theta\|_{L^p(\R^n)}^2)^{1/2},
$$
$$
\|f\|_{L^p({B_{\delta^{-1}}})}\le K_{p,n}^{(3)}(\delta)(\sum_{\theta\in \P_\delta}\|f_\theta\|_{L^p(\R^n)}^2)^{1/2},
$$
$$
\|f\|_{L^p(w_{B_{\delta^{-1}}})}\le K_{p,n}^{(4)}(\delta)(\sum_{\theta\in \P_\delta}\|f_\theta\|_{L^p(w_{B_{\delta^{-1}}})}^2)^{1/2},
$$
for each $f$ Fourier supported in $\A_{\delta}$ and each $\delta^{-1}$ ball $B_{\delta^{-1}}$, then $$ K_{p,n}^{(2)}(\delta),K_{p,n}^{(3)}(\delta),K_{p,n}^{(4)}(\delta)\sim_{n,p} K_{p,n}(\delta).$$

The same observation applies  to the family of constants related to $C_{p,n}(\delta,\nu)$ from  the multilinear inequality.
\end{remark}
\bigskip

Note that due to H\"older's inequality
$$C_{p,n}(\delta,\nu)\le K_{p,n}(\delta).$$
We will show that the reverse inequality essentially holds true.
\begin{theorem}
\label{ch ft7wbtfgb6n17r782brym9,iqmivpk[l}
 Assume one of the following holds

(i) $n=2$

(ii) $n\ge 3$ and $K_{p,d}(\delta')\lesssim_\epsilon \delta^{'-\epsilon}$ for each $\delta'$, $\epsilon>0$ and each $2\le d\le n-1$.

Then for each $0<\nu\le 1$ there is $\epsilon(\nu)$ with $\lim_{\nu\to 0}\epsilon (\nu)=0$ and $C_\nu$ such that
$$K_{p,n}(\delta)\le C_\nu \delta^{-\epsilon(\nu)}C_{p,n}(\delta,\nu)$$
for each $\delta$.
\end{theorem}

We prove the case $n=3$ and will indicate the modifications needed for $n\ge 4$. The argument will also show how to deal with the case $n=2$.

\begin{remark}
\label{markre1dd5}
If $Q_1,Q_2,Q_3\in[-1/2,1/2]^2$ the volume of parallelepiped spanned by the  unit normals to $P^2$ at $\pi^{-1}(Q_i)$ is comparable to the area of the triangle $\Delta Q_1Q_2Q_3$.
\end{remark}

The key step in the proof of Theorem \ref{ch ft7wbtfgb6n17r782brym9,iqmivpk[l} for $n=3$ is the following.

\begin{proposition}
\label{hcnyf7yt75ycn8u32r8907n580-9=--qc mvntvu5n8t}
 Assume $K_{p,2}(\delta)\lesssim_\epsilon \delta^{-\epsilon}$ for  each $\epsilon>0$. Then for each $\epsilon$ there is $C_\epsilon$ such that for each $R>1$ and $K\ge 1$
$$\|\widehat{gd\sigma}\|_{L^p(w_{B_R})}\le C_\epsilon K^{\epsilon}[(\sum_{\alpha\subset P^2\atop{\alpha:\frac{1}{K}-\text{ cap}}}\|\widehat{g_{\alpha}d\sigma}\|_{L^p(w_{B_R)}}^2)^{1/2}+(\sum_{\beta\subset P^2\atop{\beta:\frac{1}{K^{1/2}}-\text{ cap}}}\|\widehat{g_{\beta}d\sigma}\|_{L^p(w_{B_R})}^2)^{1/2}]+$$$$+K^{10}C_{p,3}(R^{-1},K^{-1})(\sum_{\Delta\subset P^2\atop{\Delta:\frac1{R^{1/2}}-\text{ cap}}}\|\widehat{g_{\Delta}d\sigma}\|_{L^p(w_{B_R})}^2)^{1/2}$$
\end{proposition}
\begin{proof}
Following the standard formalism (see for example sections 2-4 in \cite{BG}) we will regard $|\widehat{g_{\alpha}d\sigma}|$ as being essentially constant on each ball $B_K$. Denote by $c_\alpha(B_K)$ this value and  let $\alpha^*$ be the cap that maximizes it.

The starting point in the argument is the observation  that for each $B_K$ there exists a line $L=L(B_K)$ in the $(\xi_1,\xi_2)$ plane such that if
$$S_L=\{(\xi_1,\xi_2): \dist((\xi_1,\xi_2),L)\le \frac{C}{K^{\frac12}}\}$$
then for $x\in B_K$
$$ |\widehat{gd\sigma}(x)|\le $$
\begin{equation}
\label{term1.1}C\max_{\alpha}|\widehat{g_{\alpha}d\sigma}(x)|+
\end{equation}
\begin{equation}
\label{term1.4addedonJuly15/2015}C\max_{\beta}|\widehat{g_{\beta}d\sigma}(x)|+
\end{equation}
\begin{equation}
\label{term1.2}
K^{4}\max_{\alpha_1,\alpha_2,\alpha_3\atop{K^{-1}-\text{transverse}}}(\prod_{i=1}^3|\widehat{g_{\alpha_i}d\sigma}(x)|)^{1/3}+
\end{equation}\begin{equation}
\label{term1.3}
|\sum_{\beta\subset \pi^{-1}(S_L)}\widehat{g_{\beta}d\sigma}(x)|.
\end{equation}
To see this, we distinguish three scenarios.

First, if $c_\alpha(B_K)\le K^{-2}c_{\alpha^*}(B_K)$ for each $\alpha$ with  $\dist(\pi(\alpha),\pi(\alpha^*))\ge \frac{10}{K^{\frac12}}$, then the sum of \eqref{term1.1} and \eqref{term1.4addedonJuly15/2015} suffices, as
$$ |\widehat{gd\sigma}(x)|\le \sum_{\alpha:\;\dist(\pi(\alpha),\pi(\alpha^*))\ge \frac{10}{K^{\frac12}}}|\widehat{g_\alpha d\sigma}(x)|+|\sum_{\alpha:\;\dist(\pi(\alpha),\pi(\alpha^*))< \frac{10}{K^{\frac12}}}\widehat{g_\alpha d\sigma}(x)|.$$
Otherwise, there is  $\alpha^{**}$ with $\dist(\pi(\alpha^{**}),\pi(\alpha^*))\ge \frac{10}{K^{\frac12}}$ and
$c_{\alpha^{**}}(B_K)\ge K^{-2}c_{\alpha^*}(B_K)$.
The line $L$ is determined by the centers of $\alpha^*,\alpha^{**}$.

Second, if there is $\alpha^{***}$ such that  $\pi(\alpha^{***})$ intersects the complement of $S_L$ and $c_{\alpha^{***}}(B_K)\ge K^{-2}c_{\alpha^*}(B_K)$ then  \eqref{term1.2} suffices. Indeed, note that $\alpha^{*},\alpha^{**}$, $\alpha^{***}$ are $K^{-1}$ transverse by Remark \ref{markre1dd5}.

The third case is when $c_{\alpha}(B_K)< K^{-2}c_{\alpha^*}(B_K)$ whenever  $\pi(\alpha)$ intersects the complement of $S_L$. It is immediate that the sum of \eqref{term1.1} and \eqref{term1.3} will suffice in this case.

The only nontrivial case to address is the one corresponding to this latter scenario. We note that the $K^{-1}$ neighborhood of each $\beta\subset \pi^{-1}(S_L)$ is essentially a $K^{-1}\times K^{-\frac12}\times K^{-\frac12}$ box whose side of length $K^{-1}$ points in the direction of the normal vector $(2x,2y,-1)$, for some $(x,y)\in L$. These normal vectors belong to a plane with normal vector $\v$. Thus the $K^{-1}$ neighborhood of $\pi^{-1}(S_L)$ sits inside the $CK^{-1}$ neighborhood of a  cylinder in the direction $\v$, of height $\sim K^{-\frac12}$, over the parabola $\pi^{-1}(L)$.

 An application of Fubini's inequality  shows that
$$\|\sum_{\beta:\pi(\beta)\subset S_L}\widehat{g_{\beta}d\sigma}\|_{L^p(B_K)}\lesssim K_{p,2}(K^{-1})(\sum_{\beta}\|\widehat{g_{\beta}d\sigma}\|_{L^p(w_{B_K})}^2)^{1/2}.$$
We are of course relying on the fact that $\pi^{-1}(L)$ is a parabola with principal curvature roughly 1, and that the angle between the plane of the parabola and $\v$ is away from zero.

We conclude that in either case
$$\|\widehat{gd\sigma}\|_{L^p(B_K)}\le C_\epsilon K^{\epsilon}[(\sum_{\alpha\subset P^2\atop{\alpha:\frac{1}{K}\text{ cap}}}\|\widehat{g_{\alpha}d\sigma}\|_{L^p(w_{B_K})}^2)^{1/2}+(\sum_{\beta\subset P^2\atop{\beta:\frac{1}{K^{1/2}}\text{ cap}}}\|\widehat{g_{\beta}d\sigma}\|_{L^p(w_{B_K})}^2)^{1/2}]+$$$$+K^{4}\max_{\alpha_1,\alpha_2,\alpha_3\atop{K^{-1}-\text{transverse}}}\|(\prod_{i=1}^3|\widehat{g_{\alpha_i}d\sigma}(x)|)^{1/3}\|_{L^p(w_{B_K})}$$
It suffices now to raise to the $p^{th}$ power and sum over $B_K\subset B_R$ using Minkowski's inequality. Also, the norm $\|\widehat{gd\sigma}\|_{L^{p}(B_R)}$ can be replaced by the weighted norm $\|\widehat{gd\sigma}\|_{L^{p}(w_{B_R})}$ via the localization argument described in Remark \ref{gfhgdhfghdgfhdgfhdgfhertey8we837ofhgh;kjojhiop[uolip[mlo0eru89u90e98j7897989}.
\end{proof}

Rescaling gives the following.

\begin{proposition}
\label{jcfn vrfwyt8981  =,i90t8-kc0r=90,it90}
Let $\tau$ be a $\delta$ cap. Assume $K_{p,2}(\delta')\lesssim_\epsilon \delta^{'-\epsilon}$ for  each $\epsilon>0$ and $\delta'$. Then for each $\epsilon$ there is $C_\epsilon$ such that for each $R>\delta^{-2}$ and $K\ge 1$
$$\|\widehat{g_{\tau}d\sigma}\|_{L^p(w_{B_R})}\le C_\epsilon K^{\epsilon}[(\sum_{\alpha\subset \tau\atop{\alpha:\frac{\delta}{K}\text{ cap}}}\|\widehat{g_{\alpha}d\sigma}\|_{L^p(w_{B_R})}^2)^{1/2}+(\sum_{\beta\subset \tau\atop{\beta:\frac{\delta}{K^{1/2}}\text{ cap}}}\|\widehat{g_{\beta}d\sigma}\|_{L^p(w_{B_R})}^2)^{1/2}]+$$$$K^{10}C_{p,3}((R\delta^2)^{-1},K^{-1})(\sum_{\Delta\subset \tau\atop{\Delta:\frac1{R^{1/2}}\text{ cap}}}\|\widehat{g_{\Delta}d\sigma}\|_{L^p(w_{B_R})}^2)^{1/2}$$
\end{proposition}
\begin{proof}
Note that if $\gamma\subset [-1/2,1/2]^{2}$ then
$$\widehat{g_{\pi^{-1}\gamma}d\sigma}(x_1,x_2,x_3)=\int_{B_\eta(c)}\pi g(\xi_1,\xi_2)e(\xi_1x_1+\xi_2x_2+(\xi_1^2+\xi_2^2)x_3)d\xi_1d\xi_2.$$
Let $a=(a_1,a_2)$. Changing variable to $\xi_i=a_i+\delta \xi_i'$ and letting
$$\pi g^{a,\delta}(\xi')=\pi g(a+\delta \xi')$$
$$\gamma'=\delta^{-1}(\gamma-a)$$
we get
$$|\widehat{g_{\pi^{-1}\gamma}d\sigma}(x_1,x_2,x_3)|=\delta^2|\widehat{g^{a,\delta}_{\pi^{-1}\gamma'}d\sigma}(\delta(x_1+2a_1x_3),\delta(x_2+2a_2x_3),\delta^2x_3)|.$$
In particular
$$\|\widehat{g_{\pi^{-1}\gamma}d\sigma}\|_{L^p(w_{B_R})}=\delta^{2-\frac4p}\|\widehat{g^{a,\delta}_{\pi^{-1}\gamma'}d\sigma}\|_{L^p(w_{C_R})},$$
where $C_R$ is a $\sim \delta R\times \delta R\times \delta^2 R$ cylinder. Cover $C_R$ with balls $B_{\delta^2R}$. The result now follows by applying Proposition \ref{hcnyf7yt75ycn8u32r8907n580-9=--qc mvntvu5n8t}
to $g^{a,\delta}$ (with $a$ the center of $\pi(\tau)$) on each $B_{\delta^2R}$ and then summing  using Minkowski's inequality.
\end{proof}

We are now ready to prove Theorem \ref{ch ft7wbtfgb6n17r782brym9,iqmivpk[l} for $n=3$. Let $K=\nu^{-1}$. Iterate Proposition \ref{jcfn vrfwyt8981  =,i90t8-kc0r=90,it90} staring with scale $\delta=1$ until we reach scale $\delta=R^{-1/2}$.
Each iteration  lowers the scale of the caps from $\delta$ to at least $\frac{\delta}{K^{1/2}}$. Thus we have to iterate $\sim \log_K R$ times. Since $$C_{p,3}((\delta^2 R)^{-1},K^{-1})\lesssim C_{p,3}(R^{-1},\nu)$$ we get for each $\epsilon>0$
$$\|\widehat{gd\sigma}\|_{L^p(w_{B_R})}\le (CC_\epsilon K^\epsilon)^{\log_KR} K^{10}C_{p,3}(R^{-1},\nu)(\sum_{{\Delta:\frac1{R^{1/2}}\text{ cap}}}\|\widehat{g_\Delta d\sigma}\|_{L^p(w_{B_R})}^2)^{1/2}=$$
$$R^{-\log_{\nu}(CC_\epsilon)+\epsilon}\nu^{-10}C_{p,3}(R^{-1},\nu)(\sum_{{\Delta:\frac1{R^{1/2}}\text{ cap}}}\|\widehat{g_\Delta d\sigma}\|_{L^p(w_{B_R})}^2)^{1/2}.$$
The result follows since $C,C_\epsilon$ doe not depend on $\nu$.

To summarize, the proof of Theorem \ref{ch ft7wbtfgb6n17r782brym9,iqmivpk[l}  for $n=3$ relied on the hypothesis that the contribution coming from caps living near the intersection of $P^2$ with a  plane is controlled by $K_{p,2}(\delta)=O(\delta^{-\epsilon})$. When $n\ge 4$, the hypothesis $K_{p,d}(\delta)=O(\delta^{-\epsilon})$ for $2\le d\le n-1$ plays the same role, it controls the contribution coming from caps living near lower dimensional elliptic paraboloids with principal curvatures roughly 1. And of course, no such hypothesis is needed when $n=2$. The statement and the proof of Proposition \ref{jcfn vrfwyt8981  =,i90t8-kc0r=90,it90} for these values of $n$ will hold without further modifications.

\section{Proof of  Theorem \ref{thmmm1} for the paraboloid}
\label{s8}
\bigskip

In this section we prove Theorem \ref{thmmm1} for $P^{n-1}$. We first consider the open range $p>\frac{2(n+1)}{n-1}$, and in the end of the section we prove the result for the endpoint.
We use notation from the previous section such as $K_{p,n}(\delta),C_{p,n}(\delta,\nu)$ and  $\delta=N^{-1}$.

Proposition \ref{propo:parabooorescal} shows that $K_{p,n}(\delta)\lesssim K_{p,n}(\delta^{1/2})^2$. Let
.$$\gamma=\liminf_{\delta\to 0}\frac{\log K_{p,n}(\delta)}{\log(\delta^{-1})}.$$
It follows that for each $\epsilon$
$$\delta^{-\gamma}\lesssim K_{p,n}(\delta)\lesssim_\epsilon \delta^{-\gamma-\epsilon}.$$
Write $\gamma=\frac{n-1}{4}-\frac{n+1}{2p}+\alpha$. We have to show that $\alpha=0$.

The following multilinear restriction estimate from \cite{BCT} will play a key role in our proof.

\begin{theorem}
\label{invthm1}
Let $\tau_1,\ldots\,\tau_n$ be $\nu$-transverse caps on $P^{n-1}$ and assume $\widehat{f_i}$ is supported on the $\delta$-neighborhood of $\tau_i$.
Then we have
$$\|(\prod_{i=1}^n|f_i|)^{1/n}\|_{L^{\frac{2n}{n-1}}(B_N)}\lesssim_{\epsilon,\nu} N^{-\frac12+\epsilon}(\prod_{i=1}^n\|f_i\|_{L^2})^{1/n}.$$
\end{theorem}

Using Plancherel's identity this easily implies that
$$C_{p,n}(\delta,\nu)\lesssim_{\epsilon,\nu} \delta^{-\epsilon}$$
for $p=\frac{2n}{n-1}$. Combined with the Bourgain-Guth induction on scales this further leads to
$$K_{p,n}(\delta)\lesssim_\epsilon \delta^{-\epsilon}$$
for $2\le p\le \frac{2n}{n-1}$. These inequalities were proved in \cite{Bo2}. We will not rely on them in our argument below.
\bigskip

We now present the first step of our proof, which shows how to interpolate the $\|\cdot\|_{p,\delta}$ norms.

\begin{proposition}
\label{invprop1}
Let $\tau_1,\ldots\,\tau_n$ be $\nu$-transverse caps on $P^{n-1}$ and assume $\widehat{f_i}$ is supported on the $\delta$-neighborhood of $\tau_i$.
Then we have for each $\frac{2n}{n-1}\le p\le \infty$
\begin{equation}
\label{inv2}
\|(\prod_{i=1}^n|f_i|)^{1/n}\|_{L^p(B_N)}\lesssim_{\epsilon,\nu} N^{\frac{n-1}{4}-\frac{n^2+n}{2p(n-1)}+\epsilon}(\prod_{i=1}^n\|f_i\|_{\frac{p(n-1)}{n},\delta})^{\frac1{n}},
\end{equation}
and also
\begin{equation}
\label{inv6}
\|[\prod_{i=1}^n(\sum_{\theta\in \P_\delta}|f_{i,\theta}|^2)^{1/2}]^{1/n}\|_{L^p(B_N)}\lesssim_{\epsilon,\nu} N^{-\frac{n}{(n-1)p}+\epsilon}(\prod_{i=1}^n\|f_i\|_{\frac{p(n-1)}{n},\delta})^{1/n}.
\end{equation}
\end{proposition}
\begin{proof}
Let us start with the proof of  \eqref{inv2}.
Let $\lambda_n=N^{\frac{n-1}{4}}$ and $F=(\prod_{i=1}^n|f_i|)^{1/n}$.
Note that by Cauchy-Schwartz we have
$$\|F\|_\infty\le \lambda_n(\prod_{i=1}^n\|f_i\|_{\infty,\delta})^{\frac1n}.$$
By combining this with Theorem \ref{invthm1} and H\"older's inequality
we find that
\begin{equation}
\label{inv3}
\|F\|_{L^p(B_N)}\lesssim_{\epsilon,\nu} \lambda_n^{1-\frac{2n}{(n-1)p}}N^{\epsilon-\frac{n}{(n-1)p}}(\prod_{i=1}^n(\|f_i\|_{2}^{\frac{2n}{(n-1)p}}\|f_i\|_{\infty,\delta}^{1-\frac{2n}{(n-1)p}}))^{\frac1n}.
\end{equation}

Finally, to get \eqref{inv2}, we use the wave packet decomposition and the fact that
$$\|f\|_{2}^{\frac{2n}{(n-1)p}}\|f\|_{\infty,\delta}^{1-\frac{2n}{(n-1)p}}\sim\|f\|_{\frac{p(n-1)}{n},\delta}$$
if $f$ is a balanced $N$-function. We can assume $\|f_i\|_{\frac{p(n-1)}{n},\delta}=1$ for each $i$. Write like in Lemma \ref{WPD}
$$f_i=\sum_{\lambda\lesssim \|f_i\|_{\infty,\delta}}\lambda f_{i,\lambda}.$$
We use the triangle inequality to estimate the left hand side of \eqref{inv2}.
In the following $C$ will denote a large enough constant depending on $n,p$.
As $\|f_{i,\lambda}\|_\infty\lesssim 1$, we have that
$$\|f_{i,\lambda}\|_{L^p(B_N)}\le N^C.$$
As the right hand side in \eqref{inv2} is $\gtrsim N^{-C}$, it follows that the contribution coming from $\lambda f_{i,\lambda}$ with $\lambda\lesssim N^{-C}$ is well controlled.

On the other hand, recall that  by Bernstein's inequality, $\|f_i\|_{\infty,\delta}\lesssim N^{C}$. This shows that it suffices to consider $O(\log \delta^{-1})$ many terms in the triangle inequality. Each of these terms is dealt with by using \eqref{inv3}, Lemma \ref{lemmabalancednfunc} and \eqref{EE8}.

The proof of \eqref{inv6} is very similar. First, a randomization argument and Theorem \ref{invthm1} imply that
$$\|[\prod_{i=1}^n(\sum_{\theta\in \P_\delta}|f_{i,\theta}|^2)^{1/2}]^{1/n}\|_{L^{\frac{2n}{n-1}}(B_N)}\lesssim_{\epsilon,\nu}N^{-\frac12+\epsilon}
(\prod_{i=1}^n\|f_i\|_{2})^{1/n}.$$
 Combining this with the trivial inequality
$$\|[\prod_{i=1}^n(\sum_{\theta\in \P_\delta}|f_{i,\theta}|^2)^{1/2}]^{1/n}\|_{L^{\infty}(B_N)}\le
(\prod_{i=1}^n\|f_i\|_{\infty,\delta})^{1/n}$$
then with H\"older's inequality gives
$$\|[\prod_{i=1}^n(\sum_{\theta\in \P_\delta}|f_{i,\theta}|^2)^{1/2}]^{1/n}\|_{L^p(B_N)}\lesssim_{\epsilon,\nu} N^{-\frac{n}{(n-1)p}+\epsilon}(\prod_{i=1}^n(\|f\|_{2}^{\frac{2n}{(n-1)p}}\|f_i\|_{\infty,\delta}^{1-\frac{2n}{(n-1)p}}))^{1/n}.$$
Then \eqref{inv6} follows from interpolation, as explained before.
\end{proof}
\bigskip

At this point it is useful to introduce the local norms for $g:P^{n-1}\to\C$ and arbitrary balls $B$
$$\|\widehat{gd\sigma}\|_{p,\delta,B}=(\sum_{\theta:\;\delta^{1/2}-\text{cap}}\|\widehat{g_\theta d\sigma}\|_{L^p(w_{B})}^2)^{1/2}.$$
Fix $\frac{2n}{n-1}<p<\infty$. To simplify notation we also introduce the following quantities. First, define $$\xi=\frac{2}{(p-2)(n-1)}\;\;\;\;\text{  and  }\;\;\;\;\eta=\frac{n(np-2n-p-2)}{2p(n-1)^2(p-2)}.$$
For a fixed $0\le \beta\le 1$ consider the inequality
\begin{equation}
\label{inv999}
\|(\prod_{i=1}^n|\widehat{g_id\sigma}|)^{1/n}\|_{L^p(B_N)}
\lesssim_{\epsilon,\nu} A_\beta(N)N^{\epsilon}(\prod_{i=1}^n\|\widehat{g_id\sigma}\|
_{p,\delta,B_N})^{\frac{1-\beta}n}(\prod_{i=1}^n\|\widehat{g_id\sigma}\|
_{\frac{p(n-1)}{n},\delta,B_N})^{\frac\beta{n}},
\end{equation}
for arbitrary $\epsilon>0$, $\nu$, $N$, $g_i$ and $B_N$ as before.

We prove the following result.
\begin{proposition}
\label{propinv5}
(a) Inequality \eqref{inv999} holds true for $\beta=1$ with $A_1(N)=N^{\frac{n-1}{4}-\frac{n^2+n}{2p(n-1)}}$.

(b) Moreover, if we assume \eqref{inv999} for some $\beta\in (0,1]$, then we also have \eqref{inv999} for $\beta\xi$ with
$$A_{\beta\xi}(N)= A_\beta(N^{1/2})N^{\beta\eta+\frac{\gamma}{2}(1-\beta\xi)}.$$
\end{proposition}
\begin{proof}
The proof of (a) is an immediate consequence of Remark \ref{gfhgdhfghdgfhdgfhdgfhertey8we837ofhgh;kjojhiop[uolip[mlo0eru89u90e98j7897989} and \eqref{inv2}. We next focus on proving (b).
By using H\"older's inequality on the $N^{1/2}$- ball $\Delta$
$$\|\widehat{g_id\sigma}\|_{\frac{p(n-1)}{n},\delta^{1/2},\Delta}\le \|\widehat{g_id\sigma}\|_{p,\delta^{1/2},\Delta}^{1-\frac{2}{(p-2)(n-1)}}\|\widehat{g_id\sigma}\|_{2,\delta^{1/2},\Delta}^{\frac{2}{(p-2)(n-1)}},$$
we get
\begin{equation}
\label{inv4}
\|(\prod_{i=1}^n|\widehat{g_id\sigma}|)^{1/n}\|_{L^p(\Delta)}
\lesssim_{\epsilon,\nu}
N^{\epsilon}A_\beta(N^{1/2})(\prod_{i=1}^n(\|\widehat{g_id\sigma}\|_{p,\delta^{1/2},\Delta}^{1-\xi\beta}\|\widehat{g_id\sigma}\|_{2,\delta^{1/2},\Delta}^{\xi\beta})^{\frac1{n}}.
\end{equation}
 Consider a finitely overlapping cover of $B_N$ with balls $\Delta$ of radius $N^{1/2}$. Note that
 \begin{equation}
\label{harpottt4}
\|(\prod_{i=1}^n\widehat{|g_id\sigma}|)^{1/n}\|_{L^p(B_N)}\lesssim (\sum_{\Delta}\|(\prod_{i=1}^n|\widehat{g_id\sigma}|)^{1/n}\|_{L^p(\Delta)}^p)^{1/p}.
\end{equation}
We will use \eqref{inv4}  to bound each $\|(\prod_{i=1}^n|\widehat{g_id\sigma}|)^{1/n}\|_{L^p(\Delta)}$. After raising to  the $p^{th}$ power,  the right hand side of \eqref{inv4} is summed using H\"older's inequality
\begin{equation}
\label{harpottt1}
\sum_{\Delta}b_{\Delta}^{\xi\beta p}\prod_{i=1}^na_{\Delta,i}^{\frac{1-\xi\beta}{n}p}\le (\sum_{\Delta}b_{\Delta}^p)^{\xi\beta}\prod_{i=1}^n(\sum_{\Delta}a_{\Delta,i}^p)^{\frac{1-\xi\beta}{n}},
\end{equation}
with
$$a_{\Delta,i}=\|\widehat{g_id\sigma}\|_{p,\delta^{\frac12},\Delta}$$
and
$$b_{\Delta}=(\prod_{i=1}^n\|\widehat{g_id\sigma}\|_{2,\delta^{\frac12},\Delta}^p)^{\frac1{np}}.$$

To sum the factors $a_{\Delta,i}^p$ we invoke first Minkowski's inequality then Proposition \ref{propo:parabooorescal} to get
\begin{equation}
\label{harpottt2}
\sum_{\Delta}\|\widehat{g_id\sigma}\|_{p,\delta^{\frac12},\Delta}^p\lesssim\|\widehat{g_id\sigma}\|_{p,\delta^{\frac12},B_N}^p\lesssim K_{p,n}(\delta^{1/2})^{p}\|\widehat{g_id\sigma}\|_{p,\delta,B_N}^p.
\end{equation}
We next show how to sum the factors $b_{\Delta}^p$. Rather than using the $n$-linear H\"older's inequality followed by Minkowski's inequality as we did with the terms $a_{\Delta,i}^p$, we first transform  $b_{\Delta}^p$ to make it amenable to another application of  Theorem \ref{invthm1}.

To this end we recall the standard formalism (see e.g. sections 2-4 in \cite{BG}) that for each $\delta^{1/2}$-cap $\theta$,
$|\widehat{g_\theta d\sigma}|$ is essentially constant on each $\Delta$. Thus, in particular it is easy to see that
$$\sum_{\Delta\subset B_N}\prod_{i=1}^{n}\|\widehat{g_id\sigma}\|_{p,\delta,\Delta}^{\frac{p}n}\lesssim \sum_{\Delta\subset B_N}\|\prod_{i=1}^n(\sum_{\theta:\;\delta^{1/2}-\text{cap}\atop{\theta\subset\tau_i}}|\widehat{g_{i,\theta}d\sigma}|^2)^{\frac1{2n}}\|_{L^p(w_{\Delta})}^p\lesssim$$
\begin{equation}
\label{inv5}
 \lesssim\|\prod_{i=1}^n(\sum_{\theta:\;\delta^{1/2}-\text{cap}\atop{\theta\subset\tau_i}}|\widehat{g_{i,\theta}d\sigma}|^2)^{\frac1{2n}}\|_{L^p(w_{B_N})}^p.
\end{equation}
Note also that by orthogonality followed by H\"older's inequality, for each $\Delta$
\begin{equation}
\label{inv7}
\|\widehat{g_id\sigma}\|_{2,\delta^{\frac12},\Delta}\lesssim \|\widehat{g_id\sigma}\|_{2,\delta,\Delta}\lesssim N^{\frac{n}2(\frac12-\frac1p)}\|\widehat{g_id\sigma}\|_{p,\delta,\Delta}.
\end{equation}
 Now \eqref{inv7}, \eqref{inv5} and \eqref{inv6} lead to
\begin{equation}
\label{harpottt3}
\sum_{\Delta}b_\Delta^p\lesssim_{\epsilon,\nu} N^{-\frac{n}{(n-1)}+\epsilon}N^{\frac{n}{2}(\frac{p}{2}-1)}(\prod_{i=1}^n\|\widehat{g_id\sigma}\|_{\frac{(n-1)p}{n},\delta,B_N})^{1/n}.
\end{equation}
To end the argument simply invoke  estimates \eqref{harpottt4}, \eqref{harpottt1}, \eqref{harpottt2} and \eqref{harpottt3}.
\end{proof}
\bigskip

We can now present the final stage of the proof of Theorem \ref{thmmm1} for $p>\frac{2(n+1)}{n-1}$. Recall from the beginning of this section that this amounts to proving that $\alpha=0$. Since $p>\frac{2(n+1)}{n-1}$ we have that $\xi<\frac12$. We begin with a general discussion that applies in every dimension $n$, and then conclude with an inductive argument.

A simple computation reveals that the inequality $\alpha>0$ is equivalent with
$$\gamma\frac{1-\xi}{1-2\xi}>\frac{n-1}{4}-\frac{n^2+n}{2p(n-1)}+\frac{2\eta}{1-2\xi}.$$
It follows  that if $\alpha>0$ we can choose $s_0\in\N$  large enough and  $\epsilon_0$ small enough so that
$$\gamma(\frac{1-\xi}{1-2\xi}-\frac{\xi(2\xi)^{s_0}}{1-2\xi})>
$$
\begin{equation}
\label{mamatata}
\frac{n-1}{4}-\frac{n^2+n}{2p(n-1)}+2^{s_0}\epsilon_0+\frac{2\eta}{1-2\xi}(1-(2\xi)^{s_0})+\frac{n}{(n-1)p}(2\xi)^{s_0}.
\end{equation}

Choose now $\nu_0>0$ small enough such that $\epsilon_0>\epsilon(\nu_0)$, with $\epsilon(\nu_0)$ as in  Theorem \ref{ch ft7wbtfgb6n17r782brym9,iqmivpk[l}.
Note that $s_0$, $\epsilon_0$ and $\nu_0$ depend only on the fixed parameters $p,n,\alpha$. As a result, we follow our convention and do not record the dependence on them when using the symbol $\lesssim$.

Proposition \ref{propinv5} implies that for each $s\ge 0$
$$A_{\xi^s}(N)= N^{\psi(\xi^s)}$$
with
\begin{equation}
\label{inv13}
\psi(\xi^{s+1})= \frac12\psi(\xi^s)+\frac\gamma2(1-\xi^{s+1})+\eta\xi^s.
\end{equation}
Iterating \eqref{inv13}  gives
\begin{equation}
\label{inv15}
\psi(\xi^s)=\frac1{2^s}\psi(1)+\gamma(1-2^{-s})+2(\frac\eta\xi-\frac\gamma2)\frac{2^{-s}-\xi^{s}}{\xi^{-1}-2}
\end{equation}

Note that by H\"older's inequality
$$(\prod_{i=1}^n\|\widehat{g_id\sigma}\|
_{\frac{p(n-1)}{n},\delta,B_N})^{\frac1{n}}\lesssim (\prod_{i=1}^n\|\widehat{g_id\sigma}\|
_{p,\delta,B_N})^{\frac1{n}}N^{\frac{n}{(n-1)p}}.$$ Combining this with \eqref{inv999} for $\beta=\xi^s$  we get
\begin{equation}
\label{7743048765898}
C_{p,n}(\delta,\nu_0)\lesssim_{\epsilon,s} \delta^{-\epsilon}A_{\xi^s}(N)N^{\frac{n\xi^s}{(n-1)p}}.
\end{equation}
\bigskip

To finish the argument, we first consider $n=2$.
Since \eqref{7743048765898} (with $s=s_0$) holds for arbitrarily small $\delta$ and $\epsilon$, using Theorem \ref{ch ft7wbtfgb6n17r782brym9,iqmivpk[l} we get
\begin{equation}
\label{inv14}
\gamma-\epsilon_0\le \psi(\xi^{s_0})+\frac{n\xi^{s_0}}{(n-1)p}.
\end{equation}
Combining \eqref{inv15} and \eqref{inv14}  we find
$$\gamma(\frac{1-\xi}{1-2\xi}-\frac{\xi(2\xi)^{s_0}}{1-2\xi})\le \psi(1)+2^{s_0}\epsilon_0+\frac{2\eta}{1-2\xi}(1-(2\xi)^{s_0})+\frac{n}{(n-1)p}(2\xi)^{s_0},$$
which  contradicts \eqref{mamatata}, if $\alpha>0$. Thus $\alpha=0$ and Theorem \ref{thmmm1} is proved for $n=2$ and $p>6$.

The higher dimensional proof follows by induction. Assume that $n\ge 3$ and that Theorem \ref{thmmm1} was proved for all $2\le d\le n-1$ when $p>\frac{2(d+1)}{d-1}$. To prove Theorem \ref{thmmm1} in $\R^n$ for $p>\frac{2(n+1)}{n-1}$, it suffices to prove it for $\frac{2(n+1)}{n-1}<p<\frac{2n}{n-2}$. Note that in this range we have $p<\frac{2(d+1)}{d-1}$ for each $2\le d\le n-1$ and thus Theorem \ref{ch ft7wbtfgb6n17r782brym9,iqmivpk[l} is applicable, due to the induction hypothesis. To finish the argument, one applies the same reasoning as in the case $n=2$, to conclude that $\alpha>0$ leads to as contradiction.
\bigskip

It remains to see why Theorem \ref{thmmm1} holds for the endpoint $p=p_n=\frac{2(n+1)}{n-1}$. Remark \ref{gfhgdhfghdgfhdgfhdgfhertey8we837ofhgh;kjojhiop[uolip[mlo0eru89u90e98j7897989} shows that it suffices to investigate the best constant in the localized inequality
\begin{equation}
\label{inv50}
\|f\|_{L^p({B_N})}\le K_{p,n}^{(3)}(\delta)(\sum_{\theta\in \P_\delta}\|f_\theta\|_{L^p(\R^n)}^2)^{1/2},
\end{equation}
for each $N$-ball $B_N$. It suffices now to invoke Theorem \ref{thmmm1} for $p>\frac{2(n+1)}{n-1}$ together with
$$\|f\|_{L^{p_n}({B_N})}\lesssim \|f\|_{L^{p}({B_N})}N^{\frac{n}{p_n}-\frac{n}{p}}\;\;\;\text{(by H\"older's inequality)}$$
$$\|f_\theta\|_{L^p(\R^n)}\lesssim N^{\frac{n+1}{2p}-\frac{n+1}{2p_n}}\|f_\theta\|_{L^{p_n}(\R^n)}\;\;\;\text{(by Bernstein's inequality)},$$
and then to let $p\to p_n$.

\medskip

\section{Extension to other hypersurfaces}
\label{s9}

Let $S$ be a compact $C^2$ hypersurface in $\R^n$ with  positive definite second fundamental form. Recall that we have proved Theorem \ref{thmmm1} for $P^{n-1}$. By rescaling, the proof extends to elliptic paraboloids of the form
$$\{(\xi_1,\ldots,\xi_{n-1},\theta_1\xi_1^2+\ldots+\theta_{n-1}\xi_{n-1}^2)\in\R^n:\;|\xi_i|\le 1/2\},$$
with $\theta_i\in[C^{-1},C]$. The implicit bounds will of course depend on $C>0$.

We now show how to extend the result in Theorem \ref{thmmm1} to  $S$ as above. It suffices to prove the result for $p=\frac{2(n+1)}{n-1}$. We can assume that all the principal curvatures of $S$ are in $[C^{-1},C]$.

The following argument is sketched in \cite{GS} and was worked out in detail for conical surfaces in \cite{PrSe}.
For $\delta<1$, let as before $K_p(\delta)$ be the smallest constant such that for each   $f$ with Fourier support in $\A_\delta$ we have
$$\|f\|_p\le K_p(\delta)(\sum_{\theta\in \P_\delta}\|f_\theta\|_p^2)^{1/2}.$$
Fix such an $f$. First, note that
$$\|f\|_p\le K_p(\delta^{\frac23})(\sum_{\tau\in \P_{\delta^{\frac23}}}\|f_\tau\|_p^2)^{1/2}.$$
Second, our assumption on the principal curvatures of $S$ combined with Taylor's formula shows that on each $\tau\in \P_{\delta^{\frac23}}$, $S$ is within $\delta$ from a (subset of a) paraboloid with similar principal curvatures. By invoking Theorem \ref{thmmm1} for this paraboloid we get
$$\|f_\tau\|_p\lesssim_\epsilon\delta^{-\epsilon}(\sum_{\theta\in \P_\delta:\theta\subset\tau}\|f_\theta\|_p^2)^{1/2}.$$
For each $\epsilon>0$, we conclude the existence of $C_\epsilon$ such that for each $\delta<1$
$$K_p(\delta)\le C_\epsilon\delta^{-\epsilon}K_p(\delta^{\frac23}).$$
By iteration this immediately leads to $K_p(\delta)\lesssim_\epsilon \delta^{-\epsilon}$.

\section{Proof of Theorem \ref{thmmm1cone}}

To simplify notation we present the argument for $n=3$. Define the extension operator
$$E_Rg(x_1,x_2,x_3)=\int_R g(\xi_1,\xi_2)e(x_1\xi_1+x_2\xi_2+x_3\sqrt{\xi_1^2+\xi_2^2})d\xi_1d\xi_2$$
for a subset $R$ of the annulus
$$A_1=\{(\xi_1,\xi_2):1\le \sqrt{\xi_1^2+\xi_2^2}\le 2\}.$$
It will suffice to prove that
\begin{equation}
\label{jjjjeduuuueu894riefydvfcvtcrysutdgdi}
\|E_{A_1}g\|_{L^6(B_N)}\lesssim_\epsilon N^\epsilon(\sum_{S\subset A_1}\|E_{S}g\|_{L^6(w_{B_N})}^2)^{1/2},
\end{equation}
where the sum is over a tiling of $A_1$ into sectors $S$ of length 1 and aperture $N^{-1/2}$. The idea behind the proof is rather simple, we will apply the decoupling inequality  from Theorem \ref{thmmm1} to the parabola $(\xi,\frac{\xi^2}2)$. The observation that makes this application possible is the fact that
$$|(\sqrt{\xi_1^2+\xi_2^2}-\xi_1)-\frac{\xi_2^2}{2}|$$
is ``small" if $\xi_1$ is ``close" to 1 and $\xi_2$ is ``close" to 0. It remains to quantify the meaning of ``small" and ``close".

The key step is the following partial decoupling for a ``significant" subset of the cone.

\begin{proposition}
\label{prop:egfr76r234543d8ebvgtiuioerhy}
Fix $\nu,\mu>0$ such that $2\mu+\nu\le 1$ and $2\mu\ge\nu$. Given intervals $I\subset [1,2]$ and $J\subset (-\pi/2,\pi/2)$ of lengths $N^{-\nu}$ and $N^{-\mu}$ respectively,
consider the sector $$S=\{(\xi_1,\xi_2):\;\sqrt{\xi_1^2+\xi_2^2}\in I,\;\arctan (\frac{\xi_2}{\xi_1})\in J\}$$
of length $N^{-\nu}$ and aperture $N^{-\mu}$.
We have for each $\epsilon>0$
$$\|E_{S}g\|_{L^6(w_{B_N})}\lesssim_\epsilon N^{\nu\epsilon}(\sum_{S'\subset S}\|E_{S'}g\|_{L^6(w_{B_N})}^2)^{1/2},$$
where the sum runs over a tiling of $S$ into sectors $S'$ of  length $N^{-\nu}$ and aperture $\sim N^{-\mu-\frac\nu2}.$
\end{proposition}
\begin{proof}
Due to rotational and radial symmetry it suffices to assume $I=[1,1+N^{-\nu}]$ and $J=[N^{-\mu},2N^{-\mu}]$. Moreover, we may also assume that $\xi_1>0$, which implies in particular that
$$|1-\xi_1|\lesssim N^{-\nu},\,\,\,\,|\xi_2^2|\sim N^{-2\mu}.$$
Note that for each function  $F=F(x_1,x_2,x_3)$ we have
\begin{equation}
\label{rioeuvgb823087kswdjioq hufnvjhjkajemuircfgbyed56qwst7whyrnuifcvtnh089}
\|F\|_{L^6(w_{B_N})}\sim \frac1{N^{1/3}}\|\|F(x_1-y_3,x_2+y_2,x_3+y_3)\|_{L_{y_2,y_3}^6(w_{[-N,N]^2})}\|_{L^6_{x_1,x_2,x_3}(w_{B_N})}.
\end{equation}
We  apply this to $$F(x_1,x_2,x_3)=E_{S}g(x_1,x_2,x_3).$$
Fix $(x_1,x_2,x_3)\in B_N$ and evaluate the inner $L^6$ norm in \eqref{rioeuvgb823087kswdjioq hufnvjhjkajemuircfgbyed56qwst7whyrnuifcvtnh089} using the change of variables
$$(\xi_1,\xi_2)\mapsto (\eta,\xi_2):=(\sqrt{\xi_1^2+\xi_2^2}-\xi_1,\xi_2),$$
whose Jacobian $J(\eta,\xi_2)$ is nonzero. We get
$$\|\int_Sg(\xi_1,\xi_2)e(x_1\xi_1+x_2\xi_2+x_3\sqrt{\xi_1^2+\xi_2^2})e(y_2\xi_2+y_3(\sqrt{\xi_1^2+\xi_2^2}-\xi_1))d\xi_1d\xi_2\|_{L^6_{y_2,y_3}(w_{[-N,N]^2})}$$
$$=\|\int h(\eta,\xi_2)e(y_2\xi_2+y_3\eta)d\eta d\xi_2\|_{L^6_{y_2,y_3}(w_{[-N,N]^2})},$$
for some appropriate function $h$. Note that if $(\xi_1,\xi_2)\in S$
$$|\eta-\frac{\xi_2^2}{2}|=|\frac{\xi_2^2(\xi_1+\sqrt{\xi_1^2+\xi_2^2}-2)}{2(\xi_1+\sqrt{\xi_1^2+\xi_2^2})}|\lesssim N^{-2\mu-\nu}.$$
It follows that the support of $h$ is inside the $\delta\sim N^{-2\mu-\nu}$ neighborhood of the parabola
$$\{(\frac{\xi_2^2}{2},\xi_2),\;\xi_2\sim N^{-\mu}\}.$$
We can now invoke the parabolic rescaling Proposition \ref{propo:parabooorescal} and Theorem \ref{thmmm1} with $f=\widehat{h}$ to conclude that
$$\|\int h(\eta,\xi_2)e(y_2\xi_2+y_3\eta)d\eta d\xi_2\|_{L^6_{y_2,y_3}(w_{[-N,N]^2})}\lesssim_\epsilon $$$$N^{\epsilon\nu}(\sum_{|H|=N^{-\mu-\frac\nu2}}\|\int_{\xi_2\in H} h(\eta,\xi_2)e(y_2\xi_2+y_3\eta)d\eta d\xi_2\|_{L^6_{y_2,y_3}(w_{[-N,N]^2})}^2)^{1/2}.$$
The conclusion of our proposition follows now by changing back to the original variables, using Minkowski's inequality and \eqref{rioeuvgb823087kswdjioq hufnvjhjkajemuircfgbyed56qwst7whyrnuifcvtnh089}.
\end{proof}

By iterating this proposition we get the following result.

\begin{proposition}Fix $\nu=\frac1M$, with $M\ge 2$ an integer. Given intervals $I\subset [1,2]$ and $J\subset (-\pi/2,\pi/2)$ of lengths $N^{-\nu}$ and $N^{-\frac\nu2}$ respectively,
consider the sector $$S=\{(\xi_1,\xi_2):\;\sqrt{\xi_1^2+\xi_2^2}\in I,\;\arctan (\frac{\xi_2}{\xi_1})\in J\}$$
of length $N^{-\nu}$ and aperture $N^{-\frac\nu2}$.
We have
$$\|E_{S}g\|_{L^6(B_N)}\lesssim_{\epsilon} N^{\epsilon}(\sum_{\Delta\subset S}\|E_{\Delta}g\|_{L^6(w_{B_N})}^2)^{1/2},$$
where the sum runs over a tiling of $S$ into sectors $\Delta$ of length $N^{-\nu}$ and aperture $\sim N^{-\frac12}.$
\end{proposition}
\begin{proof}
Apply repeatedly Proposition \ref{prop:egfr76r234543d8ebvgtiuioerhy} with $\mu=\mu_j=j\frac{\nu}{2}$ staring with $j=1$ until $j=M-1$.
\end{proof}
\smallskip

We are left with proving that inequality \eqref{jjjjeduuuueu894riefydvfcvtcrysutdgdi} follows from this proposition.  First, note that $A_1$ can be tiled using $\sim N^{\frac{3\nu}{2}}$ sectors $S$ of length $N^{-\nu}$ and aperture $N^{-\frac\nu2}$. Call this collection $\W_\nu$. Thus, by using the Cauchy-Schwartz inequality and invoking the above proposition, we get for arbitrary $\epsilon>0$
$$
\|E_{A_1}g\|_{L^6(B_N)}\lesssim N^{\frac{3\nu}{4}}(\sum_{S\in \W_\nu}\|E_{S}g\|_{L^6(B_N)}^2)^{1/2}$$
\begin{equation}
\label{eeeefeeee1}\lesssim_{\epsilon} N^{\epsilon+\frac{3\nu}{4}}(\sum_{\Delta\subset A_1}\|E_{\Delta}g\|_{L^6(w_{B_N})}^2)^{1/2},
\end{equation}
where the sum runs over a tiling of $A_1$ into sectors $\Delta$ of length $N^{-\nu}$ and aperture $\sim N^{-\frac12}$. We also observe the following easy inequality
\begin{equation}
\label{eeeefeeee2}\|E_{\Delta}g\|_{L^6(w_{B_N})}\lesssim \|E_{S'}g\|_{L^6(w_{B_N})},
\end{equation}
where $S'$ is the sector in $A_1$ of length $1$ and aperture $N^{-1/2}$ containing $\Delta$. Note that no curvature is involved in this estimate, as the $N^{-1}$ neighborhood of $\Delta$ is essentially a rectangular parallelepiped. Since each such $S'$ contains $N^{\nu}$ sectors $\Delta$, we conclude by combining \eqref{eeeefeeee1} and \eqref{eeeefeeee2} that
$$\|E_{A_1}g\|_{L^6(B_N)}\lesssim_{\epsilon} N^{\epsilon+\frac{3\nu}{4}+\frac\nu2}(\sum_{S'\subset A_1}\|E_{S'}g\|_{L^6(w_{B_N})}^2)^{1/2}.$$
Inequality \eqref{jjjjeduuuueu894riefydvfcvtcrysutdgdi} now follows by choosing $\nu$ appropriately small.


\begin{thebibliography}{99}
\bibitem{AS} Appelbaum, R. and  Sharir, M. {\em  Repeated angles in three and four dimensions}, SIAM J. Discrete Math. 19 (2005), no. 2, 294-300 (electronic)
\bibitem{BBGR} B\'ekoll\'e, D; Bonami, A.; Garrig\'os, G. and Ricci, F. {\em Littlewood-Paley decompositions related to symmetric cones and Bergman projections in tube domains}, Proc. London Math. Soc. (3) 89 (2004), 317-360
\bibitem{BCT} Bennett, J.,  Carbery, A. and  Tao, T. {\em On the multilinear restriction and Kakeya conjectures} Acta Math. 196 (2006), no. 2, 261-302
\bibitem{Bok} Boklan, K., D. {\em The asymptotic formula in Waring's problem}  Mathematika 41 (1994), no. 2, 329-347
\bibitem{BoIw1} Bombieri, E. and Iwaniec, H. {\em On the order of $\zeta(\frac12+it)$} Ann. Scuola Norm. Sup. Pisa Cl. Sci. (4) 13 (1986), no. 3, 449-472.
\bibitem{BoIw2} Bombieri, E. and Iwaniec, H. {\em Some mean-value theorems for exponential sums} Ann. Scuola Norm. Sup. Pisa Cl. Sci. (4) 13 (1986), no. 3, 473–486
\bibitem{BoBo}   Bombieri, E. and Bourgain, J. {\em A problem on sums of two squares}, to appear in Internat. Math. Res. Notices
\bibitem{Bo1} Bourgain, J. {\em Eigenfunction bounds for the Laplacian on the n-torus}, Internat. Math. Res. Notices (1993), no. 3, 61-66.
\bibitem{Bo4} Bourgain, J. {\em On Strichartz’s inequalities and the nonlinear Schr\"odinger equation on irrational tori}, Mathematical aspects of nonlinear dispersive equations, 1-20, Ann. of Math. Stud., 163, Princeton Univ. Press, Princeton, NJ, 2007
\bibitem{Bo3} Bourgain, J. {\em   Fourier transform restriction phenomena for certain lattice subsets and applications to nonlinear evolution equations. I. Schr\"odinger equations}, Geom. Funct. Anal. 3 (1993), no. 2, 107-156
\bibitem{Bo2} Bourgain, J. {\em Moment inequalities for trigonometric polynomials with spectrum in curved hypersurfaces},  Israel J. Math. 193 (2013), no. 1, 441-458.
\bibitem{BD1} Bourgain, J. and Demeter, C. {\em Improved estimates for the  discrete Fourier restriction to the higher dimensional sphere}, Illinois J. Math. 57 (2013), no. 1, 213-227
\bibitem{BD2} Bourgain, J. and Demeter, C. {\em New bounds for the  discrete Fourier restriction to the sphere in four and five dimensions}, preprint  available on arXiv, to appear in Internat. Math. Res. Notices
\bibitem{BG} Bourgain, J. and  Guth, L. {\em Bounds on oscillatory integral operators based on multilinear estimates} Geom. Funct. Anal. 21 (2011), no. 6, 1239-1295
\bibitem{Bre} de la Bret\`eche, R. {\em R\'epartition des points rationnels sur la cubique de Segre} (French) [Distribution of rational points on Segre's cubic] Proc. Lond. Math. Soc. (3) 95 (2007), no. 1, 69-155
\bibitem{Carb}  Carbery, A. and Valdimarsson, S. I. {\em The endpoint multilinear Kakeya theorem via the Borsuk-Ulam theorem} J. Funct. Anal. 264 (2013), no. 7, 1643-1663
\bibitem{CWW} Catoire, F. and  Wang, W.-M. {\em Bounds on Sobolev norms for the defocusing nonlinear Schr\"odinger equation on general flat tori}, Commun. Pure Appl. Anal. 9 (2010), 483-491.

\bibitem{Dem23} Demeter, C. {\em Incidence theory and discrete restriction estimates},  available at arXiv:1401.1873
\bibitem{demirb} Demirbas, S. {\em Local well-posedness for 2-d Schr\"odinger equation on irrational tori and bounds on Sobolev norms}, arXiv:1307.0051
\bibitem{GS}  Garrig\'os, G. and Seeger, A. {\em A mixed norm variant of Wolff's inequality for paraboloids}, Harmonic Analysis and Partial
Differential Equations, Contemporary Mathematics, vol. 505, Amer. Math. Soc., Providence, RI, 2010, pp. 179-197
\bibitem{GarSe2} Garrig\'os, G. and Seeger, A. {\em On plate decompositions of cone multipliers} Proc. Edinb. Math. Soc. (2) 52 (2009), no. 3, 631-651

\bibitem{Gre} Greaves, G. {\em Some Diophantine equations with almost all solutions trivial}, Mathematika 44 (1997), no. 1, 14-36
\bibitem{GOW} Guo, Z., Oh, T., Wang, Y. {\em Strichartz estimates for Sch\"odinger equations on irrational tori},  arXiv:1306.4973
\bibitem{Guth} Guth, L. {\em The endpoint case of the Bennett-Carbery-Tao multilinear Kakeya conjecture}, Acta Math. 205 (2010), no. 2, 263-286
\bibitem{Jar} Jarn\'ik, V. {\em \"Uber die Gitterpunkte auf konvexen Kurven}, Math. Z. 24 (1926), pp 500-518
\bibitem{PS} Pach, J. and Sharir, M. {\em Repeated angles in the plane and related problems}, J. Combin. Theory
Ser. A, 59 (1992), pp. 12-22.
\bibitem{PrSe} Pramanik, M. and Seeger, A. {\em $L^p$ regularity of averages over curves and bounds for associated maximal operators}, Amer. J. Math. 129 (2007), no. 1, 61-103

\bibitem{RoSa} Robert, O. and Sargos, P. {\em Un th\'eor\`eme de moyenne pour les sommes d'exponentielles. Application à l'in\'egalit\'e de Weyl. (French)} [A mean value theorem for exponential sums. Application to the Weyl inequality] Publ. Inst. Math. (Beograd) (N.S.) 67(81) (2000), 14-30

\bibitem{SkW} Skinner, C. M. and  Wooley, T. D. {\em On the paucity of non-diagonal solutions in certain diagonal Diophantine systems} Quart. J. Math. Oxford Ser. (2) 48 (1997), no. 190, 255-277
\bibitem{Sog} Sogge, C. {\em Propagation of singularities and maximal functions in the plane}, Invent. Math. 104 (1991), 349-376
\bibitem{ST} Szemer\'edi, E. and Trotter, W. T. Jr. {\em Extremal problems in discrete geometry}, Combinatorica 3 (1983), no. 3-4, 381-392
\bibitem{TV} Tao, T. and Vu, V. {\em Additive combinatorics}, Cambridge Studies in Advanced Mathematics, 105. Cambridge University Press, Cambridge, 2006. xviii+512 pp.
\bibitem{Tao} Tao, T. {\em A sharp bilinear restriction estimate for paraboloids}, Geom. Funct. Anal. 13 (2003),
no. 6, 1359-1384
\bibitem{SteTom} Tomas, P. A. {\em A restriction theorem for the Fourier transform.} Bull. Amer. Math. Soc. 81 (1975), 477-478
\bibitem{VauWol}  Vaughan, R. C.; Wooley, T. D. {\em On a certain nonary cubic form and related equations}, Duke Math. J. 80 (1995), no. 3, 669-735
\bibitem{TWol} Wolff, T. {\em Local smoothing type estimates on $L^p$ for large $p$}, Geom. Funct. Anal. 10 (2000), no. 5, 1237-1288
\end{thebibliography}
\end{document}